\documentclass{amsart}
\usepackage[latin1]{inputenc}
\usepackage{subfigure}
\usepackage{color}
\usepackage{amsmath}
\usepackage{amsthm}
\usepackage{amstext}
\usepackage{amssymb}
\usepackage{amsfonts}
\usepackage{graphicx}
\usepackage{young}
\usepackage{multicol}
\usepackage{mathrsfs}
\usepackage{stmaryrd}
\usepackage[all]{xy}
\usepackage{mathabx}
\usepackage{tikz}
\usepackage{mathtools}
\usepackage{tabularx}
\usepackage{array}

\DeclareMathOperator{\Ext}{\mathrm{Ext}}
\DeclareMathOperator{\Hom}{\mathrm{Hom}}

\DeclareMathOperator{\DD}{\mathbb{D}}

\DeclareMathOperator{\AAA}{\mathbb{A}}

\DeclareMathOperator{\cD}{\mathcal{D}}
\DeclareMathOperator{\bc}{\textbf{c}}

\theoremstyle{plain}
\theoremstyle{definition}
\newtheorem{theorem}{Theorem}[section]

\newtheorem{remark}[theorem]{Remark}
\newtheorem{lemma}[theorem]{Lemma}

\newtheorem{notation}[theorem]{Notation}
\newtheorem{problem}[theorem]{Problem}

\newtheorem{example}[theorem]{Example}

\newtheorem{corollary}[theorem]{Corollary}

\DeclareMathAlphabet{\mathpzc}{OT1}{pzc}{m}{it}

\title{A combinatorial model for exceptional sequences in type A}

\author{Alexander Garver}
\address{School of Mathematics,
University of Minnesota, Minneapolis, MN 55455, USA}
\email{garv0102@math.umn.edu}

\author{Jacob P. Matherne}
\address{Department of Mathematics,
Louisiana State University, Baton Rouge, LA 70808, USA}
\email{jmath34@tigers.lsu.edu}
\thanks{The first author was supported by a Research Training Group, RTG grant DMS-1148634.}
\thanks{The second author was supported by a Graduate Assistance in Areas of National Need fellowship, GAANN grant P200A120001.}
\begin{document}

\begin{abstract}
 Exceptional sequences are certain ordered sequences of quiver representations. We use noncrossing edge-labeled trees in a disk with boundary vertices (expanding on T. Araya's work) to classify exceptional sequences of representations of Q, the linearly-ordered quiver with n vertices. We also show how to use variations of this model to classify {\bf c}-matrices of Q, to interpret exceptional sequences as linear extensions, and to give a simple bijection between exceptional sequences and certain chains in the lattice of noncrossing partitions. In the case of {\bf c}-matrices, we also give an interpretation of {\bf c}-matrix mutation in terms of our noncrossing trees with directed edges.
 \end{abstract}

\maketitle
\tableofcontents

\section{Introduction}
\label{sec:in}

Exceptional sequences are certain ordered sequences of quiver representations introduced in \cite{gr87} to study exceptional vector bundles on $\mathbb{P}^2$. Maximal such sequences called complete exceptional sequences also have connections with combinatorics as they are in bijection with maximal chains in the lattice of noncrossing partitions by the work of \cite{it09} and \cite{hk13}. More recently, complete exceptional sequences were shown to be intrinsically related to acyclic cluster algebras with principal coefficients via the work of Speyer and Thomas \cite{st13}.  They appear as certain orderings of the \textbf{c}-vectors of a $\bc$-matrix in such a cluster algebra.

We take a combinatorial approach to the study of exceptional sequences.  In \cite{a13}, Araya establishes a bijection between the set of complete exceptional collections in type $\AAA$ and the collection of certain {noncrossing spanning trees called chord diagrams}.  It is this simple combinatorial model that serves as the vehicle to our results.

In Section \ref{sec:bijection}, we decorate Araya's diagrams with edge-labelings and oriented edges so that they can keep track of both the ordering of the representations in a complete exceptional sequence as well as the signs of the rows in the $\bc$-matrix it came from.  While Araya's diagrams classify complete exceptional collections, we show that the new decorated diagrams classify more complicated objects called exceptional sequences (Theorem \ref{firstmainresult}). In this language, we give a rule for mutation of oriented diagrams arising from $\bc$-matrices in Section \ref{sec:mut}. We remark that our work is connected to that of Goulden and Yong \cite{gy02} who used edge-labeled noncrossing spanning trees in a disk to study factorizations of the long cycle $(1,2,\ldots, n+1) \in \mathfrak{S}_{n+1}$ by a {specified} collection of transpositions.

The work of Speyer and Thomas (see \cite{st13}) allows complete exceptional sequences to be obtained from $\bc$-matrices.  In \cite{onawfr13}, the number of complete exceptional sequences in type $\AAA_n$ is given, and there are more of these than there are $\bc$-matrices.  Thus, it is natural to ask exactly which $\bc$-matrices appear as Araya's diagrams.  We give an answer to this question in Section \ref{sec:cmats} 
(Theorem \ref{reachcec}).

In Section \ref{sec:pos}, we ask how many complete exceptional sequences can be formed using the representations in a complete exceptional collection. It turns out that two complete exceptional sequences can be formed in this way if they have the same underlying chord diagram without chord labels. We interpret this number as the number of linear extensions of the poset determined by the chord diagram of the complete exceptional collection. This also gives an interpretation of complete exceptional sequences as linear extensions.

With this interpretation in mind, we investigate the question of which permutations of the $\bc$-vectors of a $\bc$-matrix lead to CESs by regarding the absolute value of the $\bc$-vectors as the dimension vectors of indecomposable representations.  In Section \ref{sec: perm}, we give an explicit combinatorial description of these permutations.

In Section \ref{sec:app}, we give several applications of the theory in type $\AAA$, including combinatorial proofs that two reddening sequences produce isomorphic ice quivers (see \cite{k12} for a general proof in all types using deep category-theoretic techniques) and that there is a bijection between exceptional sequences and certain chains in the lattice of noncrossing partitions.

{\bf Acknowledgements.~}
Helpful insight was given during conversations with E. Barnard, J. Geiger, M. Kulkarni, G. Muller, G. Musiker, D. Rupel, D. Speyer, and G. Todorov.  We thank G. Muller and G. Musiker for their helpful comments on drafts of our paper and S. A. Csar whose PhD thesis defense inspired us to see the connections between exceptional sequences and linear extensions.  We also thank the 2014 Mathematics Research Communities program for giving us an opportunity to work on this exciting problem as well as for giving us a stimulating (and beautiful) place to work.

\section{Exceptional sequences}
\label{sec:prelim}

We begin by defining quivers and exchange matrices, which serve as the starting point in our study of exceptional sequences.

\subsection{Quiver mutation}\label{subsec:quivers}
A \textbf{quiver} $Q$ is a directed graph without loops or 2-cycles. In other words, $Q$ is a 4-tuple $(Q_0,Q_1,s,t)$, where $Q_0 = [m] := \{1,2, \ldots, m\}$ is a set of \textbf{vertices}, $Q_1$ is a set of \textbf{arrows}, and two functions $s, t:Q_1 \to Q_0$ defined so that for every $\alpha \in Q_1$, we have $s(\alpha) \xrightarrow{\alpha} t(\alpha)$. An \textbf{ice quiver} is a pair $(Q,F)$ with $Q$ a quiver and $F \subset Q_0$ \textbf{frozen vertices} with the additional restriction that any $i,j \in F$ have no arrows of $Q$ connecting them. We refer to the elements of $Q_0\backslash F$ as \textbf{mutable vertices}. By convention, we assume $Q_0\backslash F = [n]$ and $F = [n+1,m] := \{n+1, n+2, \ldots, m\}.$ Any quiver $Q$ can be regarded as an ice quiver by setting $Q = (Q, \emptyset)$.

The {\bf mutation} of an ice quiver $(Q,F)$ at mutable vertex $k$, denoted $\mu_k$, produces a new ice quiver $(\mu_kQ,F)$ by the three step process:

(1) For every $2$-path $i \to k \to j$ in $Q$, adjoin a new arrow $i \to j$.

(2) Reverse the direction of all arrows incident to $k$ in $Q$.

(3) Delete any $2$-cycles created during the first two steps.

\noindent We show an example of mutation below depicting the mutable (resp. frozen) vertices in black (resp. blue).
\[
\begin{array}{c c c c c c c c c}
\raisebox{.35in}{$(Q,F)$} & \raisebox{.35in}{=} & {\includegraphics[scale = .7]{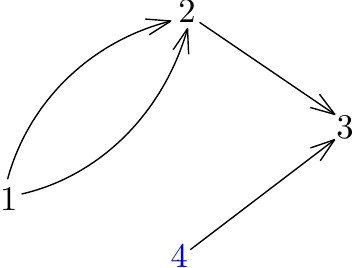}} & \raisebox{.35in}{$\stackrel{\mu_2}{\longmapsto}$} & {\includegraphics[scale = .7]{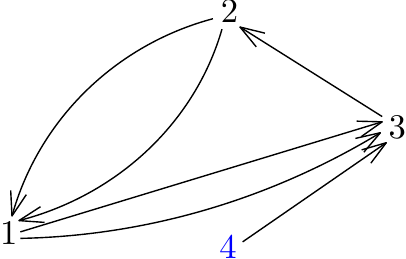}} & \raisebox{.35in}{=} & \raisebox{.35in}{$(\mu_2Q,F)$}
\end{array}
\]

The information of an ice quiver can be equivalently described by its (skew-symmetric) \textbf{exchange matrix}. Given $(Q,F),$ we define $B = B_{(Q,F)} = (b_{ij}) \in \mathbb{Z}^{n\times m} := \{n \times m \text{ integer matrices}\}$ by $b_{ij} := \#\{i \stackrel{\alpha}{\to} j \in Q_1\} - \#\{j \stackrel{\alpha}{\to} i \in Q_1\}.$ Furthermore, ice quiver mutation can equivalently be defined  as \textbf{matrix mutation} of the corresponding exchange matrix. Given an exchange matrix $B \in \mathbb{Z}^{n\times m}$, the \textbf{mutation} of $B$ at $k \in [n]$, also denoted $\mu_k$, produces a new exchange matrix $\mu_k(B) = (b^\prime_{ij})$ with entries
\[
b^\prime_{ij} := \left\{\begin{array}{ccl}
-b_{ij} & : & \text{if $i=k$ or $j=k$} \\
b_{ij} + \frac{|b_{ik}|b_{kj}+ b_{ik}|b_{kj}|}{2} & : & \text{otherwise.}
\end{array}\right.
\]
\begin{flushleft}For example, the mutation of the ice quiver above (here $m=4$ and $n=3$) translates into the following matrix mutation. Note that mutation of matrices {(or of ice quivers)} is an involution (i.e. $\mu_k\mu_k(B) = B$).\end{flushleft}
\[
\begin{array}{c c c c c c c c c c}
B_{(Q,F)} & = & \left[\begin{array}{c c c | r}
0 & 2 & 0 & 0 \\
-2 & 0 & 1 & 0\\
0 & -1 & 0 & -1\\
\end{array}\right]
& \stackrel{\mu_2}{\longmapsto} &
\left[\begin{array}{c c c | r}
0 & -2 & 2 & 0 \\
2 & 0 & -1 & 0\\
-2 & 1 & 0 & -1\\
\end{array}\right] 
& = & B_{(\mu_2Q,F)}.
\end{array}
\]

{Given a quiver $Q$, we define its \textbf{framed} (resp. \textbf{coframed}) quiver to be the ice quiver $\widehat{Q}$ (resp. $\widecheck{Q}$) where $\widehat{Q}_0\ (= \widecheck{Q}_0) := Q_0 \sqcup [n+1, 2n]$, $F = [n+1, 2n]$, and $\widehat{Q}_1 := Q_1 \sqcup \{i \to n+i: i \in [n]\}$ (resp. $\widecheck{Q}_1 := Q_1 \sqcup \{n+i \to i: i \in [n]\}$).}  Now given $\widehat{Q}$ we define the \textbf{exchange tree} of $\widehat{Q}$, denoted $ET(\widehat{Q})$, to be the (a priori infinite) graph whose vertices are quivers obtained from $\widehat{Q}$ by a finite sequence of mutations and with two vertices connected by an edge if and only if the corresponding quivers are obtained from each other by a single mutation. Similarly, define the \textbf{exchange graph} of $\widehat{Q}$, denoted $EG(\widehat{Q})$, to be the quotient of $ET(\widehat{Q})$ where two vertices are identified if and only if there is a \textbf{frozen isomorphism} of the corresponding quivers (i.e. an isomorphism that fixes the frozen vertices). Such an isomorphism is equivalent to a simultaneous permutation of the rows and columns of the corresponding exchange matrices.

Given $\widehat{Q}$, we define the \textbf{c}-\textbf{matrix} $C(n) = C_R(n)$ (resp. $C = C_R$) of $R \in ET(\widehat{Q})$ (resp. $R \in EG(\widehat{Q})$)  to be the submatrix of $B_R$ where $C(n) := (b_{ij})_{i \in [n], j \in [n+1, 2n]}$ (resp. $C := (b_{ij})_{i \in [n], j \in [n+1,2n]}$). We let \textbf{c}-mat($Q$) $:= \{C_R: R \in EG(\widehat{Q})\}$. By definition, $B_R$ (resp. $C$) is only defined up to simultaneous permutations of its rows and columns (resp. up to permutations of its rows) for any $R \in EG(\widehat{Q})$.

A row vector of a \textbf{c}-matrix, $\overrightarrow{c}$, is known as a \textbf{c}-\textbf{vector}. The celebrated theorem of Derksen, Weyman, and Zelevinsky \cite[Theorem 1.7]{dwz10}, known as the {sign-coherence} of $\textbf{c}$-vectors, states that for any $R \in ET(\widehat{Q})$ and $i \in [n]$ the \textbf{c}-vector $\overrightarrow{c_i}$ is a nonzero element of $\mathbb{Z}_{\ge 0}^n$ or $\mathbb{Z}_{\le0}^n$. Thus we say a \textbf{c}-vector is either \textbf{positive} or \textbf{negative}.

For the purposes of this paper, we will only be concerned with the linearly-ordered $\AAA_n$ quiver $Q$ shown below with its framed quiver $\widehat{Q}$.
\[{Q} = \raisebox{.175in}{\begin{xy} 0;<1pt,0pt>:<0pt,-1pt>:: 
(0,10) *+{1} ="0",
(30,10) *+{2} ="1",
(60,10) *+{\cdots} ="2",
(130,10) *+{n} ="4",
(100,10) *+{n-1} ="5",
"1", {\ar"0"},
"2", {\ar"1"},
"5", {\ar "2"},
"4", {\ar"5"},
\end{xy}}
\ \ \ \ \ \ \ \ \ \widehat{Q} = \begin{xy} 0;<1pt,0pt>:<0pt,-1pt>:: 
(0,10) *+{1} ="0",
(30,10) *+{2} ="1",
(60,10) *+{\cdots} ="2",
(100,10) *+{n-1} ="3",
(130,10) *+{n} ="4",
(0,-20) *+{\textcolor{blue}{\text{$n+1$}}} ="5",
(30,-20) *+{\textcolor{blue}{\text{$n+2$}}} ="6",
(100,-20) *+{\textcolor{blue}{\text{$2n-1$}}} ="7",
(130,-20) *+{\textcolor{blue}{\text{$2n$}}} ="8",
"1", {\ar"0"},
"0", {\ar"5"},
"2", {\ar"1"},
"1", {\ar"6"},
"3", {\ar"2"},
"4", {\ar"3"},
"3", {\ar"7"},
"4", {\ar"8"},
\end{xy}\]

\subsection{Representations of quivers}

A \textbf{representation} $V = ((V_i)_{i \in Q_0}, (\varphi_\alpha)_{\alpha \in Q_1})$ of a quiver $Q$  is an assignment of a $k$-vector space $V_i$ to each vertex $i$ and a $k$-linear map $\varphi_\alpha: V_{s(\alpha)} \rightarrow V_{t(\alpha)}$ to each arrow $\alpha$ where $k$ is a field.  The \textbf{dimension vector} of $V$ is the vector $\underline{\dim}(V):=(\dim V_i)_{i\in Q_0}$.
Here is an example of a representation, with $\underline{\dim}(V) = (3,3,2)$, of the \textbf{mutable part} of the quiver depicted in Section \ref{subsec:quivers}.
\[
{\includegraphics[scale = .8]{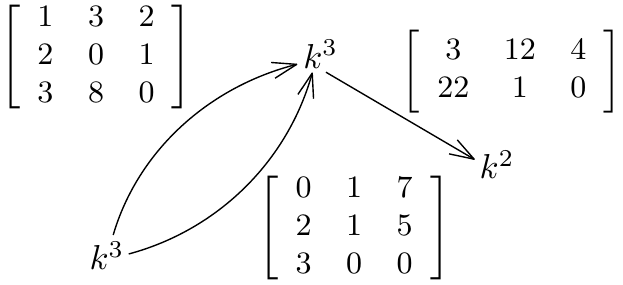}}
\]
Let $V = ((V_i)_{i \in Q_0}, (\varphi_\alpha)_{\alpha \in Q_1})$ and $W  = ((W_i)_{i \in Q_0}, (\varrho_\alpha)_{\alpha \in Q_1})$ be two representations of a quiver $Q$. A \textbf{morphism} $\theta : V \rightarrow W$ consists of a collection of linear maps $\theta_i : V_i \rightarrow W_i$ that are compatible with each of the linear maps in $V$ and $W$.  That is, for each arrow $\alpha \in Q_1$, we have $\theta_{t(\alpha)} \circ \varphi_\alpha = \varrho_\alpha \circ \theta_{s(\alpha)}$.  An \textbf{isomorphism} of quiver representations is a morphism $\theta: V \to W$ where $\theta_i$ is a $k$-vector space isomorphism for all $i \in Q_0$. We define $V \oplus W := ((V_i\oplus W_i)_{i \in Q_0}, (\varphi_\alpha \oplus \varrho_\alpha)_{\alpha \in Q_1})$ to be the \textbf{direct sum} of $V$ and $W$. We say that a nonzero representation $V$ is \textbf{indecomposable} if it is not isomorphic to a direct sum of two nonzero representations. 

For the linearly-ordered $\AAA_n$ quiver, it is a standard exercise
to show that its indecomposable representations up to isomorphism are exactly the representations $X_{i,j}$ with $0 \le i < j \le n$ defined as
\[
\begin{array}{c c c c c c c c c c c c c c c c c c c c c c c}
& & 1 & & & & i & & & & & & j & & & & & & n \\
X_{i,j} & := &0 & \stackrel{0}{\longleftarrow} & \cdots & \stackrel{0}{\longleftarrow} & 0 & \stackrel{0}{\longleftarrow} & k & \stackrel{1}{\longleftarrow} & \cdots & \stackrel{1}{\longleftarrow} & k & \stackrel{0}{\longleftarrow} & 0 & \stackrel{0}{\longleftarrow} & \cdots & \stackrel{0}{\longleftarrow} & 0.
\end{array}
\]

We remark that representations of $Q$ can equivalently be regarded as modules over the \textbf{path algebra} $kQ$. As such, one can define $\Ext_{kQ}^s(V,W)$ ($s \ge 0$) and $\Hom_{kQ}(V,W)$ for any representations $V$ and $W$.  We refer the reader to \cite{ass06} for more details on representations of quivers.

An \textbf{exceptional sequence} $\xi = (V_1,\ldots, V_k)$  ($k \le n:= \#Q_0$) is an ordered \textit{list} of \textbf{exceptional representations} $V_j$ of $Q$ (i.e. $\Ext_{kQ}^s(V_j,V_j) = 0$ for all $s\ge 1$) satisfying $\Hom_{kQ}(V_j,V_i) = 0$ and $\Ext_{kQ}^s(V_j,V_i) = 0$ if $i < j$ for all $s \ge 1$. We define an \textbf{exceptional collection} $\overline{\xi} = \{V_1,\ldots, V_k\}$ to be a \textit{set} of exceptional representations $V_j$ of $Q$ that can be ordered in such a way that they define an exceptional sequence. When $k = n$, we say $\xi$ (resp. $\overline{\xi}$) is a \textbf{complete exceptional sequence} (CES) (resp. \textbf{complete exceptional collection} (CEC)). For the linearly-ordered quiver, a representation is exceptional if and only if it is indecomposable.

The following result of Speyer and Thomas gives a beautiful connection between $\bc$-matrices of an acyclic quiver $Q$ and CESs.  It serves as motivation for our work. Before stating it we remark that for any $R \in ET(\widehat{Q})$ and any $i \in [n]$ where $Q$ is an acyclic quiver, the $\bc$-vector $\overrightarrow{c_i} = \overrightarrow{c_i}(R) = \pm \underline{\dim}(V_i)$ for some indecomposable representation of $Q$ (see \cite{c12}).

\begin{notation}
Let $\overrightarrow{c}$ be a \textbf{c}-vector of an acyclic quiver $Q$. Define 
$$\begin{array}{rcl}
|\overrightarrow{c}| & := & \left\{\begin{array}{rcl} \overrightarrow{c} & : & \text{if $\overrightarrow{c}$ is positive}\\ -\overrightarrow{c} & : & \text{if $\overrightarrow{c}$ is negative.} \end{array}\right.
\end{array}$$
\end{notation}

\begin{theorem}\cite{st13}\label{st}
Let $C \in \textbf{c}$-mat$(Q)$, let $\{\overrightarrow{c_i}\}_{i \in [n]}$ denote the \textbf{c}-vectors of $C$, and let $|\overrightarrow{c_i}| = \underline{\dim}(V_i)$ for some indecomposable representation of $Q$. There exists a permutation $\sigma \in \mathfrak{S}_n$ such that  $(V_{\sigma(1)},...,V_{\sigma(n)})$ is a CES with the property that if there exist positive \textbf{c}-vectors in $C$, then there exists $k \in [n]$ such that $\overrightarrow{c_{\sigma(i)}}$ is positive if and only if $i \in [k,n]$, and $\Hom_{kQ}(V_i,V_j)=0$ if $\overrightarrow{c_i}, \overrightarrow{c_j}$ have the same sign.  Conversely, any set of $n$ vectors having these properties defines a $\bc$-matrix whose rows are $\{\overrightarrow{c_i}\}_{i\in[n]}$.
\end{theorem}

\section{Chord diagrams and exceptional sequences}
\label{sec:bijection}
We henceforth fix $n \in \mathbb{N}$ and use $Q$ to denote the linearly-ordered $\mathbb{A}_n$ quiver. 

A \textbf{chord diagram} $d = \{c(i_\ell,j_\ell)\}_{\ell \in [k]}$ with $k \in [n]$ is a graph embedded in a disk with $n+1$ vertices, called \textbf{marked points}, on the boundary of the disk. The marked points are labeled $0,1,...,n$ counterclockwise starting from the top of the {disk}. We let $c(i,j)$ denote the \textbf{chord} connecting the marked points labeled $i$ and $j$. Furthermore, we require that the chords are pairwise non-intersecting in the interior of the disk, that the underlying graph of $d$ is a forest,
 and that each pair of distinct marked points has at most one chord connecting them. We similarly define a \textbf{labeled} (resp. \textbf{oriented}) chord diagram, denoted $d(k) = \{(c(i_\ell, j_\ell), s_\ell)\}_{\ell \in [k]}$ (resp. $\overrightarrow{d} = \{\overrightarrow{c}(i_\ell, j_\ell)\}_{\ell \in [k]}$), to be a chord diagram each of whose chords are labeled by an integer $s_\ell \in [k]$ bijectively (resp. whose chords $\overrightarrow{c}(i_\ell, j_\ell)$ are oriented from $i_\ell$ to $j_\ell$). We give examples of each type of chord diagram of type $\AAA_3$ below.

\[
{\includegraphics[scale = .4]{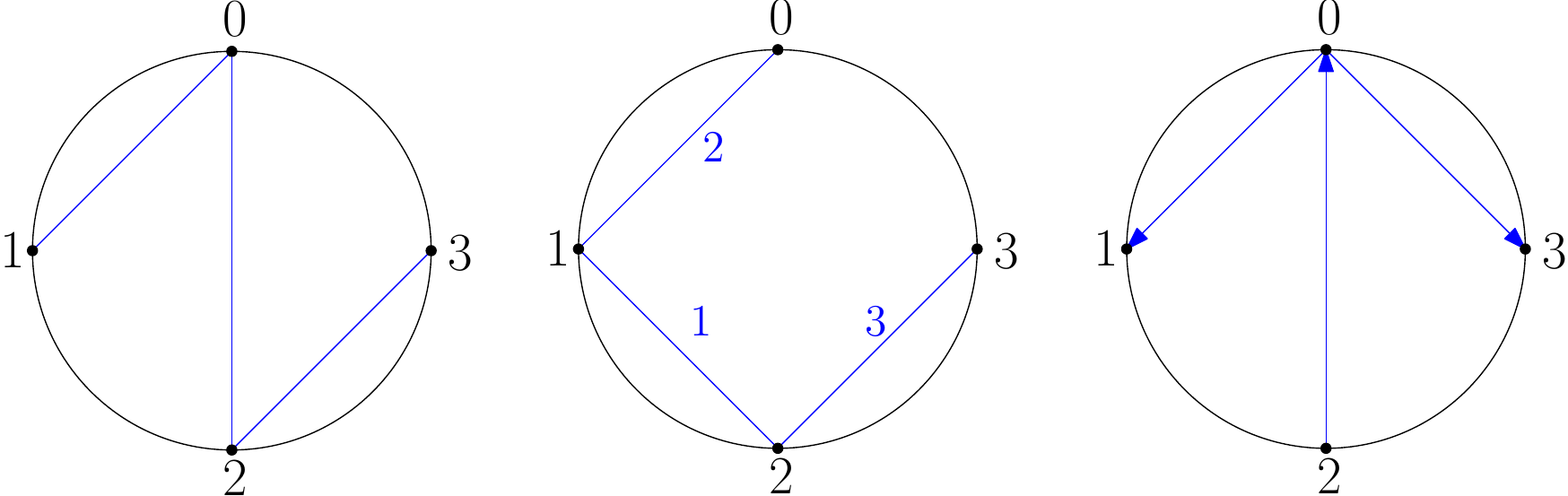}}
\]

\begin{theorem}\cite[Thm. 1.1]{a13}\label{a13}
Let $\overline{\mathcal{E}} := \{\text{complete exceptional collections}\}$, $\mathcal{D} := \{\text{diagrams } d = \{c(i_\ell,j_\ell)\}_{\ell \in [n]}\}$. The map $\Phi: \mathcal{\overline{E}} \to \mathcal{D}$ defined by $\overline{\xi} = \{X_{i_1,j_1},\ldots, X_{i_n,j_n}\} \mapsto \{c(i_\ell,j_\ell)\}_{\ell \in [n]}$ is a bijection.
\end{theorem}

Let $d(k)$ be a labeled diagram, let $i$ be any marked point, and let $((c(i, j_{j_1}), s_1), \ldots, (c(i, j_{j_r}), s_r))$ be the complete list of chords in $d(k)$ involving $i$ ordered so that chords appearing later in the list are clockwise from those earlier in the list in $d(k)$. We say the chord-labeling of $d(k)$ is \textbf{good} if for each marked point $i$ one has $s_1 < \cdots < s_r$. The labeling of the diagram in the previous example is not good.

We now state our first main result. Let $\mathcal{D}(k)$ denote the set of diagrams with $k$ chords and with good labelings and let $\mathcal{E}(k) := \{\text{exceptional sequences } \xi = (V_1,\ldots, V_k)\}.$

\begin{theorem}\label{firstmainresult}
Let $k \in [n]$. One has a bijection $\widetilde{\Phi}: \mathcal{E}(k) \to \mathcal{D}(k)$ given by $${\xi} = (X_{i_1,j_1},\ldots, X_{i_k,j_k}) \mapsto \{(c(i_\ell,j_\ell), k+1 - \ell)\}_{\ell \in [k]}.$$
\end{theorem}

Before giving the proof of Theorem~\ref{firstmainresult}, we present several technical lemmas. We first mention an obvious lemma followed by a nontrivial result from representation theory of finite dimensional algebras.

\begin{lemma}
A sequence of representations $(V_1,\ldots, V_k)$ is exceptional if and only if $(V_i,V_j)$ is an \textbf{exceptional pair} (i.e. an exceptional sequence of length 2) for all $i,j \in [k]$ {with} $i < j$. 
\end{lemma}

\begin{theorem}\cite[Cor. 2.14]{ass06}\label{arformula}
Let $A$ be a $k$-algebra and $M,N$ be two $A$-modules. If $\text{pd}M \le 1$ and $N$ is arbitrary, then there exists a $k$-linear isomorphism $$\Ext^1_A(M,N) \cong D\Hom_A(N,\tau M)$$ where $\tau$ denotes the \textbf{Auslander-Reiten translation} and $D(-) := \Hom_k(-, k)$ denotes the standard duality functor on the category of $k$-vector spaces.
\end{theorem}
\begin{remark}
The path algebra $kQ$ is a hereditary algebra so $\text{pd}(M) \le 1$ for any $kQ$-module $M$. Furthermore, we have that
$$\begin{array}{rcl}
\tau X_{i,j} & = & \left\{\begin{array}{lcl} X_{i-1,j-1} & : & \text{if $i > 0$}\\ 0 & : & \text{if $i = 0$.} \end{array}\right.\\ \end{array}$$
\end{remark}

The next three lemmas are a crucial part of the proof of Theorem~\ref{firstmainresult}. They are needed to prove that the chord-labeling of $\widetilde{\Phi}(\xi)$, the labeled diagram determined by an exceptional sequence, is good. The proof of Lemma~\ref{last} is similar to that of Lemma~\ref{first} so we only prove the latter.

\begin{lemma}\label{first}
Let $0\le i < j < j^\prime \le n$. Then $(X_{i,j}, X_{i,j^\prime})$ is the unique exceptional pair defined by $X_{i,j}$ and $X_{i,j^\prime}$.
\end{lemma}

\begin{proof}
Observe that $\Hom_{kQ}(X_{i,j}, X_{i,j^\prime}) \neq 0$ so $(X_{i,j^\prime}, X_{i,j})$ is not an exceptional pair. Note that an indecomposable representation $X_{k,\ell}$ is projective if and only if $k = 0.$  On the other hand, one checks that $\Hom_{kQ}(X_{i,j^\prime}, X_{i,j}) = 0.$ Similarly, by Theorem~\ref{arformula} one has
$$\begin{array}{rcl}
\Ext^1_{kQ}(X_{i,j^\prime}, X_{i,j}) & \cong & D\Hom_{kQ}(X_{i,j}, \tau X_{i,j^\prime})\\
& = & \left\{\begin{array}{lcl}D\Hom_{kQ}(X_{i,j}, X_{i-1, j^\prime - 1}) & : & \text{if $i > 0$}\\
D\Hom_{kQ}(X_{i,j}, 0) & : & \text{if $i = 0$} \end{array}\right.\\
& = & 0.
\end{array}$$
Thus $(X_{i,j}, X_{i,j^\prime})$ is the unique exceptional pair defined by $X_{i,j}$ and $X_{i,j^\prime}$.
\end{proof}

\begin{lemma}\label{middle}
Let $0\le i < j < j^\prime \le n$. Then $(X_{j,j^\prime}, X_{i,j})$ is the unique exceptional pair defined by $X_{i,j}$ and $X_{j,j^\prime}$.
\end{lemma}
\begin{proof}
Observe that by Theorem~\ref{arformula}, one has
$$\begin{array}{rcl}
\Ext^1_{kQ}(X_{j,j^\prime}, X_{i,j}) & \cong & D\Hom_{kQ}(X_{i,j}, \tau X_{j,j^\prime})\\
& = & D\Hom_{kQ}(X_{i,j}, X_{j-1, j^\prime - 1})\\
& \neq & 0.
\end{array}$$
Thus $(X_{i,j}, X_{j,j^\prime})$ is not an exceptional pair.

On the other hand, one checks that $\Hom_{kQ}(X_{i,j}, X_{j,j^\prime}) = 0$. Similarly, by Theorem~\ref{arformula} one has
$$\begin{array}{rcl}
\Ext^1_{kQ}(X_{i,j}, X_{j,j^\prime}) & \cong & D\Hom_{kQ}(X_{j,j^\prime}, \tau X_{i,j})\\
& = & \left\{\begin{array}{lcl}D\Hom_{kQ}(X_{j,j^\prime}, X_{i-1, j - 1}) & : & \text{if $i > 0$}\\
D\Hom_{kQ}(X_{j,j^\prime}, 0) & : & \text{if $i = 0$} \end{array}\right.\\
& = & 0.
\end{array}$$
Thus $(X_{j,j^\prime}, X_{i,j})$ is the unique exceptional pair defined by $X_{i,j}$ and $X_{j,j^\prime}$.
\end{proof}

\begin{lemma}\label{last}
Let $0 \le i < i^\prime < j \le n$. Then $(X_{i,j}, X_{i^\prime,j})$ is the unique exceptional pair defined by $X_{i,j}$ and $X_{i^\prime, j}$.
\end{lemma}

The next lemma characterizes indecomposable representations that determine exactly two exceptional pairs.

\begin{lemma}\label{separatechords}
The representations $X_{i,j}$ and $X_{k,\ell}$ define two distinct exceptional pairs if and only if $c(i,j)$ and $c(k,\ell)$ have no interior point and no marked points in common.
\end{lemma}
\begin{proof}
Without loss of generality, we can assume that there are the following two cases. 

$\begin{array}{ccccc}
a) & 0\le i < j < k < \ell \le n & \text{or} \\
b) & 0 \le k < i < j < \ell \le n.
\end{array}$

\noindent The proof of the statement in \textit{Case b)} is similar to that of \textit{Case a)} so we omit it. 

\textit{Case a)} Clearly, $\Hom_{kQ}(X_{i,j}, X_{k,\ell}) = 0$ and $\Hom_{kQ}(X_{k,\ell}, X_{i,j}) = 0$. Now observe that from Theorem~\ref{arformula}, we have
$$\begin{array}{rcl}
\Ext^1_{kQ}(X_{i,j}, X_{k,\ell}) & \cong & D\Hom_{kQ}(X_{k,\ell}, \tau X_{i,j})\\
& = & \left\{\begin{array}{lcl}D\Hom_{kQ}(X_{k,\ell}, X_{i-1, j - 1}) & : & \text{if $i > 0$}\\
D\Hom_{kQ}(X_{k,\ell}, 0) & : & \text{if $i = 0$} \end{array}\right.\\
& = & 0.
\end{array}$$
Using Theorem~\ref{arformula} again, we obtain
$$\begin{array}{rcl}
\Ext^1_{kQ}(X_{k,\ell}, X_{i,j}) & \cong & D\Hom_{kQ}(X_{i,j}, \tau X_{k,\ell})\\
& = & D\Hom_{kQ}(X_{i,j}, X_{k-1,\ell - 1}) \\
& = & 0.
\end{array}$$
Thus both $(X_{i,j}, X_{k,\ell})$ and $(X_{k,\ell}, X_{i,j})$ are exceptional pairs.
\end{proof}

The next lemma shows that $\widetilde{\Phi}(\xi)$ has no crossing chords.

\begin{lemma}\label{cross}
The representations $X_{i,j}$ and $X_{k,\ell}$ do not define an exceptional pair if  and only if $c(i,j)$ and $c(k,\ell)$ intersect in the interior of the disk.
\end{lemma}
\begin{proof}
Without loss of generality, we can assume that $0\le i < k < j < \ell \le n$. Observe that $\Hom_{kQ}(X_{i,j}, X_{k,\ell}) \neq 0$ so $(X_{k,\ell}, X_{i,j})$ is not an exceptional pair. By Theorem~\ref{arformula}, one has
$$\begin{array}{rcl}
\Ext^1_{kQ}(X_{k,\ell}, X_{i,j}) & \cong & D\Hom_{kQ}(X_{i,j}, \tau X_{k,\ell})\\
& = & D\Hom_{kQ}(X_{i,j}, X_{k-1, \ell - 1})\\
& \neq & 0.
\end{array}$$
Thus $(X_{i,j}, X_{k,\ell})$ is not an exceptional pair.
\end{proof}

\begin{remark}
Lemmas~\ref{first}, \ref{middle}, \ref{last}, \ref{separatechords}, and \ref{cross} can be deduced from \cite[Lemma 3.2]{a13}.
\end{remark}

\begin{proof}[Proof of Theorem~\ref{firstmainresult}]
Let $\xi \in \mathcal{E}(k)$. By Lemma~\ref{cross}, $\widetilde{\Phi}(\xi)$ has no crossing chords. Let $(V_1,V_2)$ be an exceptional pair appearing in $\xi$ with $V_i$ corresponding to chord $c_i$ in $\widetilde{\Phi}(\xi)$ for $i = 1,2$ where $c_1$ and $c_2$ intersect only at a marked point. Note that by the definition of $\widetilde{\Phi}(\xi)$, the chord label of $c_1$ is bigger than that of $c_2$. From Lemmas~\ref{first}, \ref{middle}, and \ref{last}, chord $c_1$ appears clockwise from $c_2$ in $\widetilde{\Phi}(\xi)$. Thus the chord-labeling of $\widetilde{\Phi}(\xi)$ is good so $\widetilde{\Phi}(\xi) \in \mathcal{D}(k)$ for any $\xi \in \mathcal{E}(k)$.

Let $\widetilde{\Psi}: \mathcal{D}(k) \to \mathcal{E}(k)$ be defined by $\{(c(i_\ell, j_\ell), \ell)\}_{\ell \in [k]} \mapsto (X_{i_k,j_k}, X_{i_{k-1}, j_{k-1}}, \ldots, X_{i_1,j_1})$. We will show that $\widetilde{\Psi}(d(k)) \in \mathcal{E}(k)$ for any $d(k) \in \mathcal{D}(k)$ and that $\widetilde{\Psi} = \widetilde{\Phi}^{-1}.$ Let $\widetilde{\Psi}(\{(c(i_\ell, j_\ell), \ell)\}_{\ell \in [k]}) = (X_{i_k,j_k}, X_{i_{k-1}, j_{k-1}},\ldots, X_{i_1, j_1}).$ Consider the pair $(X_{i_s,j_s}, X_{i_{s^\prime}, j_{s^\prime}})$ with $s > s^\prime.$ We will show that $(X_{i_s,j_s}, X_{i_{s^\prime}, j_{s^\prime}})$ is an exceptional pair and thus conclude that $\widetilde{\Psi}(\{(c(i_\ell, j_\ell), \ell)\}_{\ell \in [k]}) \in \mathcal{E}(k)$ for any $d(k) \in \mathcal{D}(k)$. Clearly, $c(i_s,j_s)$ and $c(i_{s^\prime}, j_{s^\prime})$ are noncrossing. If $c(i_s, j_s)$ and $c(i_{s^\prime}, j_{s^\prime})$ do not intersect at a marked point, then by Lemma~\ref{separatechords} $(X_{i_s,j_s}, X_{i_{s^\prime}, j_{s^\prime}})$ is exceptional. Now suppose $c(i_s, j_s)$ and $c(i_{s^\prime}, j_{s^\prime})$ intersect at a marked point. Because the chord-labeling of $\{(c(i_\ell, j_\ell), \ell)\}_{\ell \in [k]}$ is good, $c(i_s, j_s)$ is clockwise from  $c(i_{s^\prime}, j_{s^\prime})$. By Lemmas~\ref{first}, \ref{middle}, and \ref{last}, we have that $(X_{i_s,j_s}, X_{i_{s^\prime}, j_{s^\prime}})$ is exceptional.

To see that $\widetilde{\Psi} = \widetilde{\Phi}^{-1}$, observe that
$$\begin{array}{rcl}
\widetilde{\Phi}\left(\widetilde{\Psi}(\{(c(i_\ell, j_\ell), \ell)\}_{\ell \in [k]})\right) & = &\widetilde{\Phi}\left((X_{i_k,j_k}, X_{i_{k-1}, j_{k-1}},\ldots, X_{i_1,j_1})\right)\\
& = & \{(c(i_\ell, j_\ell), k+1 - (k+1 - \ell))\}_{\ell \in [k]} \\
& = & \{(c(i_\ell, j_\ell), \ell)\}_{\ell \in [k]}.
\end{array}$$
Thus $\widetilde{\Phi}\circ \widetilde{\Psi} = 1_{\mathcal{D}(k)}.$ Similarly, one shows that $\widetilde{\Psi}\circ \widetilde{\Phi} = 1_{\mathcal{E}(k)}.$ Thus $\widetilde{\Phi}$ is a bijection.
\end{proof}

\section{Mutation of oriented chord diagrams}
\label{sec:mut}

In this section, we will develop a process of mutation for oriented diagrams that is analogous to the previously defined mutation for quivers and exchange matrices.  The following lemma gives a way to associate an oriented diagram to any given $\bc$-matrix.

\begin{lemma}\label{cmatdiag}
Let $C \in \textbf{c}$-mat$(Q)$ with $\{\overrightarrow{c_i}\}_{i \in [n]}$ its \textbf{c}-vectors. Let $\overrightarrow{c_i} = \pm\underline{\dim}(X_{i_1,i_2})$ where the sign of $\overrightarrow{c_i}$ is determined by $C$. There is an injective map $\textbf{c}\text{-mat}(Q)\to \overrightarrow{\mathcal{D}} := \{\text{oriented diagrams } \overrightarrow{d} = \{\overrightarrow{c}(i_\ell, j_\ell)\}_{\ell \in [n]}\}$ given by 
$$\overrightarrow{c_i} \mapsto \left\{\begin{array}{r c l c c c c c} \overrightarrow{c}(i_1,i_2) & : & \text{\overrightarrow{c_i} is positive}\\ \overrightarrow{c}(i_2,i_1) & : & \text{\overrightarrow{c_i} is negative.}
\end{array}\right.$$
Thus each \textbf{c}-matrix $C$ determines a unique oriented diagram denoted $\overrightarrow{d}_C$ with $n$ oriented chords.
\end{lemma}

\begin{proof}
By Theorem~\ref{st}, any \textbf{c}-matrix $C$ defines a unique CEC $\overline{\xi}$ and therefore by Theorem~\ref{a13} defines a unique chord diagram $d_C := \{c(j_\ell, k_\ell)\}_{\ell \in [n]} = \Phi(\overline{\xi}).$ Observe that for any $\ell \in [n]$ there exists a \textbf{c}-vector $\overrightarrow{c_i}$ in $C$ such that $c(j_\ell, k_\ell) = c(i_1, i_2)$ where $\overrightarrow{c_i} = \pm\underline{\dim}(X_{i_1,i_2})$ where the sign of $\overrightarrow{c_i}$ is determined by $C$. Orient the chords of $d_C$ according to the rule 
$${c(i_1,i_2)} \mapsto \left\{\begin{array}{r c l c c c c c} \overrightarrow{c}(i_1,i_2) & : & \text{\overrightarrow{c_i} is positive}\\ \overrightarrow{c}(i_2,i_1) & : & \text{\overrightarrow{c_i} is negative.}
\end{array}\right.$$
Define $\overrightarrow{d}_C$ to be the resulting oriented chord diagram. This clearly defines the desired injective map.
\end{proof}

\begin{remark}\label{natmap}
There is a natural map
$$\begin{array}{rcl}
\overrightarrow{\mathcal{D}} & \stackrel{|\cdot|}{\longrightarrow} & \mathcal{D}\\
\overrightarrow{d} = \{\overrightarrow{c}(i_\ell, j_\ell)\}_{\ell \in [n]} & \longmapsto & |\overrightarrow{d}| := \{c(i_\ell, j_\ell)\}_{\ell \in [n]}.
\end{array}$$
\end{remark}

The next lemma will be helpful in describing mutation of oriented diagrams in this section.
\begin{lemma}\label{shareendpt}
Let $R \in EG(\widehat{Q})$, let $B = (b_{ij})$ denote its exchange matrix with $\bc$-matrix $C$, and let $\overrightarrow{d}_{C}$ denote the oriented diagram associated to $C$.  If $b_{kj} \neq 0$ for some $k,j \in [n]$, then the oriented chords $\overrightarrow{c_k}(i_1(k), i_2(k))$ and $\overrightarrow{c_j}(i_1(j), i_2(j))$ of $\overrightarrow{d}_{C}$ corresponding to $k$ and $j$ share an endpoint.
\end{lemma}
\begin{lemma}\cite[Prop. 2.4]{bv08}\label{BV}
A quiver $R$ (without frozen vertices) is mutation equivalent to $Q$ if and only if $R$ satisfies the following:\\
$\begin{array}{rrl}
      \textit{i)} &  & \text{All non-trivial cycles in the underlying graph of $R$ are oriented and of}\\
      & & \text{length 3.} \\
      \textit{ii)} & & \text{Any vertex has at most four neighbors.} \\ 
      \textit{iii)} & & \text{If a vertex has four neighbors, then two of its adjacent arrows belong to}\\
      & & \text{one 3-cycle, and the other two belong to another 3-cycle.}\\
      \textit{iv)} & & \text{If a vertex has exactly three neighbors, then two of its adjacent arrows}\\
      & & \text{belong to a 3-cycle, and the third arrow does not belong to any 3-cycle.}\\
     \end{array}$ 
\end{lemma}
\begin{corollary}\label{coro:onearrow}
If $R$ is mutation equivalent to $Q$, then any two mutable vertices of $R$ have at most one arrow connecting them.
\end{corollary}

\begin{proof}[Proof of Lemma~\ref{shareendpt}]
As $B$ is skew-symmetric, it is enough to assume that $b_{kj} = 1$ (using {Corollary \ref{coro:onearrow}}) with $k<j$. We can also assume that $\overrightarrow{c_j}(R)$ is positive where $R$ is the ice quiver whose exchange matrix is $B$. The proof when $\overrightarrow{c_j}(R)$ is negative is very similar.

Suppose that $\overrightarrow{c_k}(R) = \underline{\dim}(X_{i_1(k),i_2(k)})$ and $\overrightarrow{c_j}(R) = \underline{\dim}(X_{i_1(j), i_2(j)}).$ Then $$\overrightarrow{c_k}(\mu_jR) = \underline{\dim}(X_{i_1(k), i_2(k)}) + \underline{\dim}(X_{i_1(j), i_2(j)}).$$ In order to have $\overrightarrow{c_k}(\mu_jR) = \underline{\dim}(X_{s,t})$ for some $s, t \in [n]$ with $s < t$, we must have $i_1(k) = i_2(j)$ or $i_2(k) = i_1(j)$. Thus the oriented chords $\overrightarrow{c_k}(i_1(k), i_2(k))$ and $\overrightarrow{c_j}(i_1(j), i_2(j))$ of $\overrightarrow{d}_{C}$ share an endpoint.

Now suppose that $\overrightarrow{c_k}(R) = -\underline{\dim}(X_{i_1(k),i_2(k)})$ and $\overrightarrow{c_j}(R) = \underline{\dim}(X_{i_1(j), i_2(j)}).$ Then we have $$\overrightarrow{c_j}(\mu_kR) = -\underline{\dim}(X_{i_1(k), i_2(k)}) + \underline{\dim}(X_{i_1(j), i_2(j)}).$$ In order to have $\overrightarrow{c_j}(\mu_kR) = \pm \underline{\dim}(X_{s,t})$ for some $s, t \in [n]$ with $s < t$, we must have $i_1(k) = i_1(j)$ so that 
$$\begin{array}{rcl}\overrightarrow{c_j}(\mu_kR) & = &\left\{\begin{array}{r c l c c c c c} \underline{\dim}(X_{i_2(k),i_2(j)}) & : & \text{if $i_2(k) < i_2(j)$}\\ -\underline{\dim}(X_{i_2(j),i_2(k)}) & : & \text{if $i_2(j) < i_2(k)$}
\end{array}\right.
\end{array}$$
or $i_2(k) = i_2(j)$ so that 
$$\begin{array}{rcl}\overrightarrow{c_j}(\mu_kR) & = &\left\{\begin{array}{r c l c c c c c} -\underline{\dim}(X_{i_1(k),i_1(j)}) & : & \text{if $i_1(k) < i_1(j)$}\\ \underline{\dim}(X_{i_1(j),i_1(k)}) & : & \text{if $i_1(j) < i_1(k)$.}
\end{array}\right.
\end{array}$$
In either case, the oriented chords $\overrightarrow{c_k}(i_1(k), i_2(k))$ and $\overrightarrow{c_j}(i_1(j), i_2(j))$ of $\overrightarrow{d}_{C}$ share an endpoint.
\end{proof}

We have found that oriented diagrams arising from $\bc$-matrices via Lemma \ref{cmatdiag} have certain chords which share an endpoint via Lemma \ref{shareendpt}.  We now keep track of how oriented diagrams are affected when their corresponding $\bc$-matrices are mutated (via the mutation of exchange matrices in Section \ref{subsec:quivers}).  Just as for quivers, mutation of diagrams is a local property--that is, the chords which do not share an endpoint with the mutating chord are not affected by mutation.  The way that intersecting chords are affected depends on exactly which endpoint they share.  In the theorem below, we give one of our main results--a mutation formula for oriented diagrams arising from $\bc$-matrices.

\begin{theorem}\label{mainmutation}
Let $R \in EG(\widehat{Q})$, let $B = (b_{ij})$ denote its exchange matrix with $\bc$-matrix $C$, and let $\overrightarrow{d}_{C}$ denote the oriented diagram associated to $C$. Then for any $k \in [n]$, $\overrightarrow{d}_{\mu_kC}$ is obtained from $\overrightarrow{d}_C$ by the following {two-step} process \\
$\begin{array}{rlllcc}
i) & \text{if $j \in [n]$, $b_{kj} \neq 0$, and $\text{sgn}(b_{kj}) \neq \text{sgn}(\overrightarrow{c_k}(R))$, replace $\overrightarrow{c_j}(i_1(j), i_2(j))$ with $\overrightarrow{c_{j}}((i_1(j))^\prime, (i_2(j))^\prime)$}\\
& \text{where}\\
& \begin{array}{rl}a) & (i_1(j))^\prime = i_2(k), \ (i_2(j))^\prime = i_2(j) \ \text{ if } \ i_1(k) = i_1(j),\\
b) & (i_1(j))^\prime = i_1(j), \ (i_2(j))^\prime = i_1(k) \ \text{ if } \ i_2(k) = i_2(j), \\
c) & (i_1(j))^\prime = i_1(k), \ (i_2(j))^\prime = i_2(j) \ \text{ if } \ i_2(k) = i_1(j), \\
d) & (i_1(j))^\prime = i_1(j), \ (i_2(j))^\prime = i_2(k) \ \text{ if } \ i_1(k) = i_2(j).
\end{array}\\
ii) & \text{replace $\overrightarrow{c_k}(i_1(k), i_2(k))$ with $\overrightarrow{c_k}(i_2(k), i_1(k))$.}
\end{array}$
\end{theorem}

\begin{proof}
As $B$ is skew-symmetric, it is enough to assume that $b_{kj} = 1$ (using {Corollary \ref{coro:onearrow}}) with $k<j$. Since we need to prove that $\overrightarrow{d}_{\mu_kC}$ is obtained from $\overrightarrow{d}_C$ using the operation described in $i)$, we further assume that $\overrightarrow{c_k}(R)$ is negative where $R$ is the ice quiver whose exchange matrix is $B$. Without loss of generality, we assume that $\overrightarrow{c_k}(R) = -\underline{\dim}(X_{i_1(k),i_2(k)})$ and $\overrightarrow{c_j}(R) = \text{sgn}(\overrightarrow{c_j}(R))\underline{\dim}(X_{i_1(j),i_2(j)})$ so $i_1(k)<i_2(k)$ and $i_1(j) < i_2(j)$. Note that by Lemma~\ref{shareendpt}, we know that the oriented chord determined by $\overrightarrow{c_k}(R)$, namely $\overrightarrow{c_k}(i_2(k), i_1(k))$, and the oriented chord determined by $\overrightarrow{c_j}(R)$ share an endpoint in $\overrightarrow{d}_C$. We have 
$$\begin{array}{cclccc}
\overrightarrow{c_j}(\mu_kR) & = & \text{sgn}(\overrightarrow{c_j}(R))\overrightarrow{c_j}(R) - \overrightarrow{c_k}(R) \\
& = & \text{sgn}(\overrightarrow{c_j}(R))\underline{\dim}(X_{i_1(j),i_2(j)}) - \underline{\dim}(X_{i_1(k),i_2(k)}) \\
& = & \left\{\begin{array}{r c l c c c c c} \underline{\dim}(X_{i_1(j),i_2(j)}) - \underline{\dim}(X_{i_1(k),i_2(k)}) & : & \text{$\overrightarrow{c_j}(R)$ is positive}\\ -\underline{\dim}(X_{i_1(j),i_2(j)}) - \underline{\dim}(X_{i_1(k),i_2(k)}) & : & \text{$\overrightarrow{c_j}(R)$ is negative.}
\end{array}\right.
\end{array}$$

Assume $\overrightarrow{c_j}(R)$ is positive. Thus the oriented chord determined by $\overrightarrow{c_j}(R)$ is $\overrightarrow{c_j}(i_1(j),i_2(j))$. In order to have $\overrightarrow{c_j}(\mu_kR) = \pm\underline{\dim}(X_{s,t})$ for some $s,t \in [n]$ with $s<t$, we must have either $i_1(j) = i_1(k)$ or $i_2(j) = i_2(k)$. Now observe that we have
$$\begin{array}{cclccc}
\overrightarrow{c_j}(\mu_kR) & = & \left\{\begin{array}{r c l c c c c c}
-\underline{\dim}(X_{i_2(j),i_2(k)}) & : & \text{if $i_1(j) = i_1(k)$ and $i_2(j) < i_2(k)$}\\
\underline{\dim}(X_{i_2(k),i_2(j)}) & : & \text{if $i_1(j) = i_1(k)$ and $i_2(k) < i_2(j)$}\\
\underline{\dim}(X_{i_1(j),i_1(k)}) & : & \text{if $i_2(j) = i_2(k)$ and $i_1(j) < i_1(k)$}\\
-\underline{\dim}(X_{i_1(k),i_1(j)}) & : & \text{if $i_2(j) = i_2(k)$ and $i_1(k) < i_1(j)$.}
\end{array}\right.
\end{array}$$
We see that the oriented chord determined by $\overrightarrow{c_j}(\mu_kR)$ in $\overrightarrow{d}_{\mu_kC}$ is $\overrightarrow{c_j}(i_2(k), i_2(j))$ if $i_1(j) = i_1(k)$ or $\overrightarrow{c_j}(i_1(j),i_1(k))$ if $i_2(j) = i_2(k)$ as desired.

Now assume that $\overrightarrow{c_j}(R)$ is negative. Then the oriented chord determined by $\overrightarrow{c_j}(R)$ is $\overrightarrow{c_j}(i_2(j),i_1(j))$. In order to have $\overrightarrow{c_j}(\mu_kR) = \pm\underline{\dim}(X_{s,t})$ for some $s,t\in [n]$ with $s<t$, we must have either $i_1(j) = i_2(k)$ or $i_2(j) = i_1(k)$. Now observe that we have
$$\begin{array}{cclccc}
\overrightarrow{c_j}(\mu_kR) & = & \left\{\begin{array}{r c l c c c c c}
-\underline{\dim}(X_{i_1(k),i_2(j)}) & : & \text{if $i_1(j) = i_2(k)$}\\
-\underline{\dim}(X_{i_1(j),i_2(k)}) & : & \text{if $i_2(j) = i_1(k)$}
\end{array}\right.
\end{array}$$
We see that the oriented chord determined by $\overrightarrow{c_j}(\mu_kR)$ in $\overrightarrow{d}_{\mu_kC}$ is $\overrightarrow{c_j}(i_2(j), i_1(k))$ if $i_1(j) = i_2(k)$ or $\overrightarrow{c_j}(i_2(k),i_1(j))$ if $i_2(j) = i_1(k)$ as desired.

It is clear that to obtain $\overrightarrow{d}_{\mu_kC}$ from $\overrightarrow{d}_{C}$ one replaces $\overrightarrow{c_k}(i_2(k), i_1(k))$ with $\overrightarrow{c_k}(i_1(k), i_2(k))$ as $\overrightarrow{c_k}(\mu_kR) = -\overrightarrow{c_k}(R) = \underline{\dim}(X_{i_1(k),i_2(k)}).$ Thus $\overrightarrow{d}_{\mu_kC}$ is obtained from $\overrightarrow{d}_C$ by the process described in the statement of the theorem.
\end{proof}

\begin{remark}
Our notion of mutation expands on Araya's notion of mutation of exceptional pairs, which is an operation on unoriented chord diagrams (see \cite[Lemma 5.2]{a13}).
\end{remark}
\begin{remark}
The process of oriented diagram mutation requires the information of the mutable part of the ice quiver $R$.
\end{remark}

\begin{example}We conclude this section with an example of mutation of oriented chord diagrams associated to $\bc$-matrices in type $\AAA_3$. Letting $R = \widehat{Q}$ with exchange matrix $B$, we first apply $\mu_1$. We use both $i)$ and $ii)$ in Theorem \ref{mainmutation} to obtain $\overrightarrow{d}_{\mu_1C}$. Note that the chord $\overrightarrow{c}(2,3)$ corresponding to {$\overrightarrow{c_3}(R)$} was not affected. Next, we apply $\mu_3$ to $B^\prime = \mu_1B$. Since $\text{sgn}(b^\prime_{32})= 1 = \text{sgn}(\overrightarrow{c_3}(\mu_1R))$, we only use $ii)$ from Theorem~\ref{mainmutation} to obtain $\overrightarrow{d}_{\mu_3\mu_1C}$.  Notice that the chord $\overrightarrow{c}(0,1)$ corresponding to {$\overrightarrow{c_1}(\mu_1 R)$} was not affected.
\[\begin{xy} 0;<1pt,0pt>:<0pt,-1pt>:: 
(0,10) *+{1} ="0",
(30,10) *+{2} ="1",
(60,10) *+{3} ="2",
(0,-20) *+{\textcolor{blue}{\text{$4$}}} ="5",
(30,-20) *+{\textcolor{blue}{\text{$5$}}} ="6",
(60,-20) *+{\textcolor{blue}{\text{$6$}}} ="7",
"1", {\ar"0"},
"0", {\ar"5"},
"2", {\ar"1"},
"1", {\ar"6"},
"2", {\ar"7"},
\end{xy} \stackrel{\mu_1}{\longmapsto}\begin{xy} 0;<1pt,0pt>:<0pt,-1pt>:: 
(0,10) *+{1} ="0",
(30,10) *+{2} ="1",
(60,10) *+{3} ="2",
(0,-20) *+{\textcolor{blue}{\text{$4$}}} ="5",
(30,-20) *+{\textcolor{blue}{\text{$5$}}} ="6",
(60,-20) *+{\textcolor{blue}{\text{$6$}}} ="7",
"0", {\ar"1"},
"5", {\ar"0"},
"2", {\ar"1"},
"1", {\ar"6"},
"2", {\ar"7"},
"1", {\ar"5"},
\end{xy} \stackrel{\mu_3}{\longmapsto} \begin{xy} 0;<1pt,0pt>:<0pt,-1pt>:: 
(0,10) *+{1} ="0",
(30,10) *+{2} ="1",
(60,10) *+{3} ="2",
(0,-20) *+{\textcolor{blue}{\text{$4$}}} ="5",
(30,-20) *+{\textcolor{blue}{\text{$5$}}} ="6",
(60,-20) *+{\textcolor{blue}{\text{$6$}}} ="7",
"0", {\ar"1"},
"5", {\ar"0"},
"1", {\ar"2"},
"1", {\ar"6"},
"7", {\ar"2"},
"1", {\ar"5"},
\end{xy}\]
\begin{center}\raisebox{-0.5\height}{\includegraphics[scale = .5]{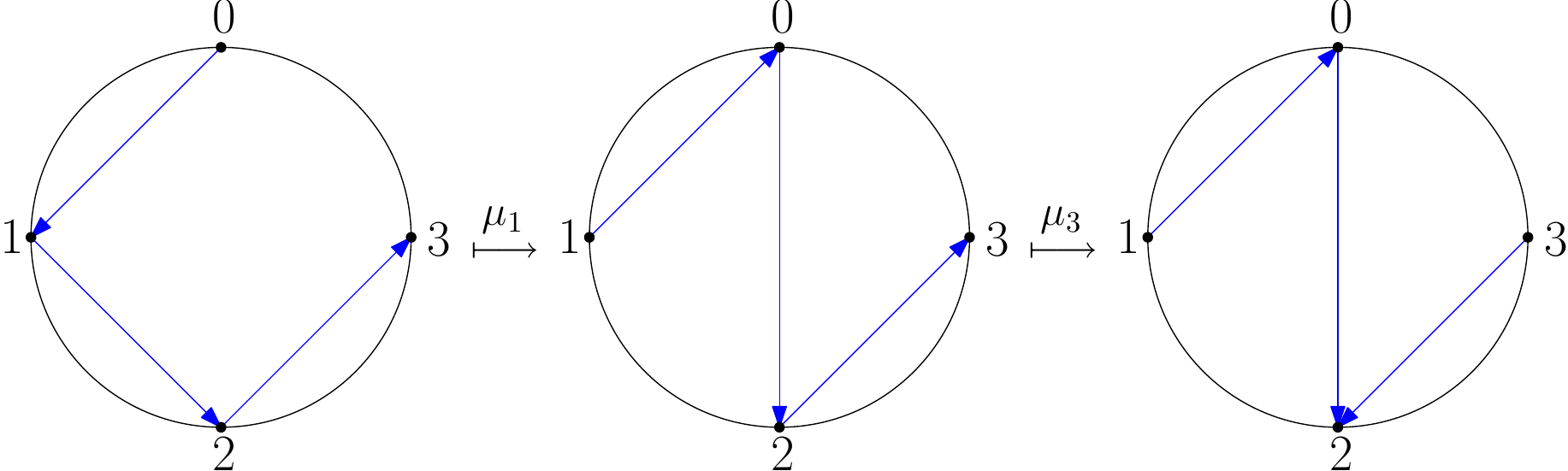}}\end{center}\end{example}

\section{Reachable CECs and classification of $\bc$-matrices}
\label{sec:cmats}
The bijection in Theorem \ref{a13} associates to every CEC $\overline{\xi}$ a corresponding chord diagram $d$.  We note that in $\AAA_3$, oriented versions of the following three diagrams are not obtainable by diagram mutation.
\begin{center}
\includegraphics[scale = .4]{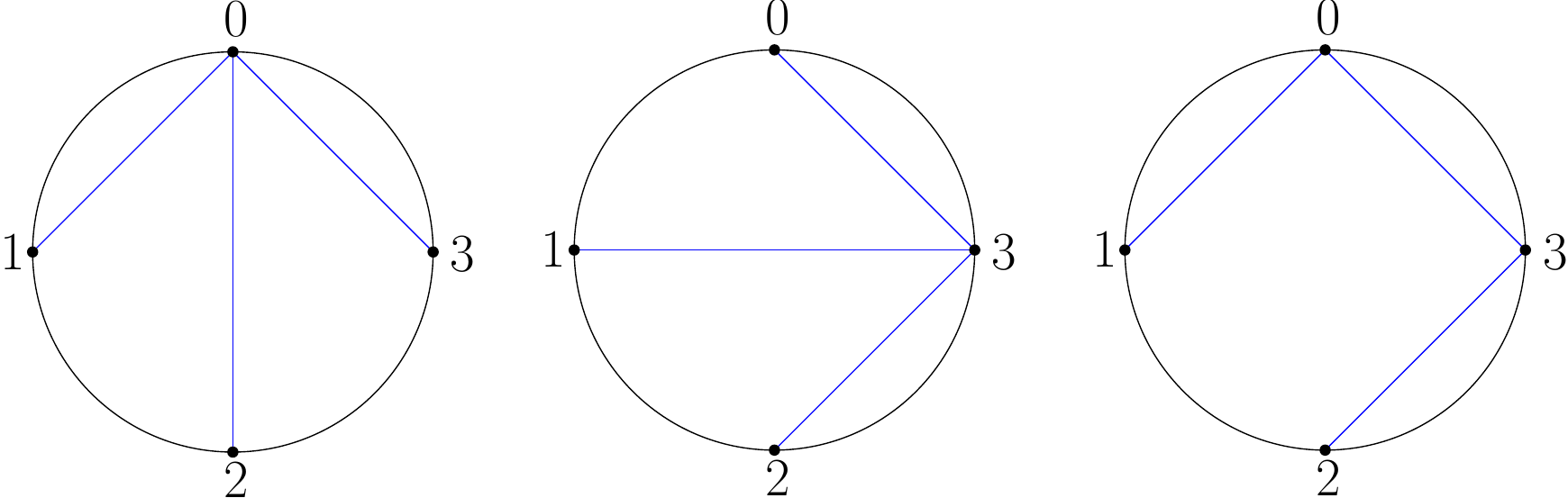}
\end{center}
This behavior occurs in type $\AAA_n$ for every $n \geq 3$.  In this section, we classify exactly which CECs appear by classifying diagrams instead.  Along the way, we obtain a diagrammatic classification of $\bc$-matrices.

A CEC $\overline{\xi}$ is said to be \textbf{reachable} if there exists an orientation of the diagram $d$ of $\overline{\xi}$ (see Theorem \ref{a13}), denoted $\overrightarrow{d}$, such that $\overrightarrow{d} = \overrightarrow{d}_C$ for some $C \in \textbf{c}$-mat$(Q)$.  Let $\overrightarrow{d} \in \overrightarrow{\mathcal{D}}$ and let $k \in [0,n]$ denote a marked point on the disk. Define the following subsets of the vertices of the disk
$$\begin{array}{ccccccc}
\text{out}(\overrightarrow{d},k) & := & \{i \in [0,n]: \overrightarrow{c}(k,i) \in \overrightarrow{d} \} \nonumber & \text{ and } & \text{in}(\overrightarrow{d}, k) & := & \{j \in [0,n]: \overrightarrow{c}(j,k) \in \overrightarrow{d}\}. \nonumber
\end{array}$$

\begin{lemma}\label{orientconfigs}
Let $C \in \textbf{c}$-mat$(Q)$, let $k \in [0,n]$ be any marked point of the disk, and let $I:=\{i_1< i_2 < \cdots < i_m\} \subset [0,n]$ denote the vertices of the disk that are connected to $k$ in $\overrightarrow{d}_C$. Then $\overrightarrow{d}_C$ is \textbf{weakly separated} at $k$ in the sense that there exists $s,t \in [m]$ such that 
$\text{in}(\overrightarrow{d}_C,k) = \{i_s < i_{s+1} < \cdots < i_{t-1} < i_t\}$  and $\text{out}(\overrightarrow{d}_C,k) = I\backslash\text{in}(\overrightarrow{d}_C,k)$ or one of $\text{out}(\overrightarrow{d}_C,k)$  or $\text{in}(\overrightarrow{d}_C,k)$ is empty.
\end{lemma}
\begin{proof}
Assume that both $\text{out}(\overrightarrow{d}_C,k)$ and $\text{in}(\overrightarrow{d}_C,k)$ are nonempty. Let $k \in [0,n]$ be a marked point of the disk. Observe that $$\text{in}(\overrightarrow{d}_C,k) = \underbracket{\{i \in \text{in}(\overrightarrow{d}_C,k): i < k\}}_{=:B_1} \sqcup \underbracket{\{i \in \text{in}(\overrightarrow{d}_C,k): i > k\}}_{=:B_2}.$$ The oriented chords of $\overrightarrow{d}_C$ that begin at elements of $\text{out}(\overrightarrow{d}_C,k)$ and $B_1$ correspond to positive $\textbf{c}$-vectors in $C$ while the oriented chords of that begin at elements of $B_2$ correspond to negative $\textbf{c}$-vectors in $C$. 

By Theorem~\ref{st}, there exists a positive integer $j$ and a complete exceptional sequence $$\xi_C = (E_n, E_{n-1}, \ldots, E_{n-j}, E_{n-j-1}, \ldots, E_1)$$ such that the \textbf{c}-vectors corresponding to each of $E_n, \ldots, E_{n-j}$ (resp. $E_{n-j-1}, \ldots, E_1$) are negative (resp. positive). Since both $\text{out}(\overrightarrow{d}_C,k)$ and $\text{in}(\overrightarrow{d}_C,k)$ are nonempty, we must have $j \in [n-2]$. Let $d(n) := \widetilde{\Phi}(\xi_C) \in \mathcal{D}(n)$ denote the labeled chord diagram corresponding to $\xi_C$. Note that $d(n)$ and $\overrightarrow{d}_C$ have the same underlying (unoriented, unlabeled) chord diagrams so we can form the labeled, oriented diagram $\overrightarrow{d}(n)_C$ by adding the edge labels of $d(n)$ to the chords of $\overrightarrow{d}_C$. By Theorem~\ref{firstmainresult}, the maximum length sequence $\{\ell_t\}_{t \in [m]}$ of edge labels obtained by reading counterclockwise around $k$ in the interior of the disk is decreasing. Thus there exists $t^\prime \in [m]$ such that all chords connected to $k$ with label $\ell_{t}$ for $t \le t^\prime$ correspond to negative \textbf{c}-vectors in $C$ and all chords connected to $k$ with label $\ell_t$ for $t \ge t^\prime + 1$ correspond to positive \textbf{c}-vectors in $C$. 

To see that $\text{in}(\overrightarrow{d}_C,k) = \{i_s < i_{s+1} < \cdots < i_{t-1} < i_t\}$ for some $s,t \in [m]$, it is enough to show that there does not exist $i_r \in \text{out}(\overrightarrow{d}_C,k)$ such that $i_{r-1}, i_{r+1} \in \text{in}(\overrightarrow{d}_C,k)$.  Suppose that such an $i_r$ exists. The following diagram appears as a subdiagram of $\overrightarrow{d}_C$ where $0\le i_{r-1} < i_r < i_{r+1} \le n.$ 

\begin{center}
\includegraphics[scale=1]{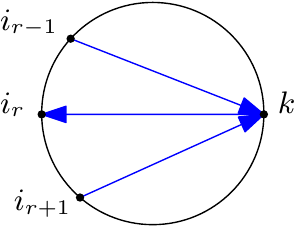}
\end{center}

\noindent However, this contradicts the existence of $t^\prime \in [m]$ such that all chords connected to $k$ with label $\ell_{t}$ for $t \le t^\prime$ correspond to negative \textbf{c}-vectors in $C$ and all chords connected to $k$ with label $\ell_t$ for $t \ge t^\prime + 1$ correspond to positive \textbf{c}-vectors in $C$. The result follows.\end{proof}

Lemma~\ref{orientconfigs} gives some slight restrictions on the orientation of chords in oriented diagrams arising from $\bc$-matrices.  The following theorem greatly reduces the number of possible configurations.

\begin{theorem}\label{reachcec}
A CEC $\overline{\xi}$ is reachable if and only if there exists an orientation of the chords of its associated diagram $d$, denoted $\overrightarrow{d}$, so that $\overrightarrow{d}$ is weakly separated at each marked point and $\overrightarrow{d}$ has no subdiagram of any of the following forms where we assume $0 \le i < j < t \le n$ and $0 \le s < i < j \le n$:
\[
{\includegraphics[scale = .8]{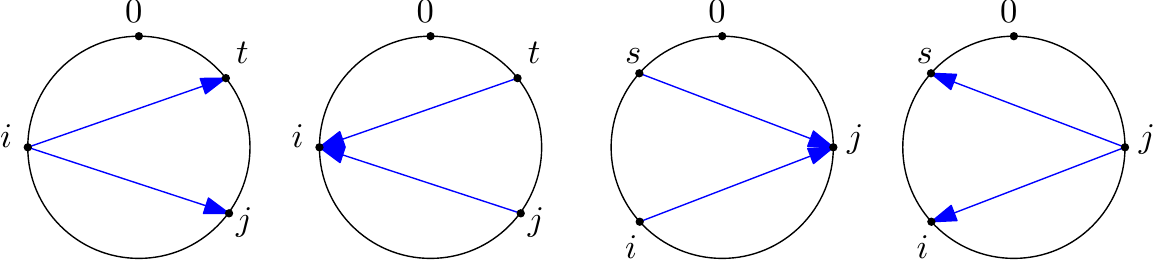}}
\]
\end{theorem}
\begin{remark}\label{homnonvan}
Each of the oriented diagrams of the form appearing in the statement of the theorem satisfies the conditions of Lemma~\ref{first} and Lemma~\ref{last}. In each case, it is easy to see that the representations $E_1$ and $E_2$ corresponding to the two chords have the property that either $\Hom_{kQ}(E_1,E_2) \neq 0$ or $\Hom_{kQ}(E_2,E_1) \neq 0$.
\end{remark}

\begin{proof}[Proof of Theorem~\ref{reachcec}]
Let $\overline{\xi}$ be a reachable CEC. Let $\overrightarrow{d}_C$ denote the oriented diagram associated to $\overline{\xi}$. By Lemma~\ref{orientconfigs}, it is clear that $\overrightarrow{d}_C$ is weakly separated at each marked point. By Remark~\ref{homnonvan} and Theorem~\ref{st}, it is clear that $\overrightarrow{d}_C$ contains none of the oriented diagrams of the form shown above as subdiagrams.

Conversely, suppose $\overline{\xi}$ is a CEC with the properties in the statement of the theorem. Let $\overrightarrow{d} = \{\overrightarrow{c}(j_\ell, k_\ell)\}_{\ell \in [n]}$ denote the oriented diagram corresponding to $\overline{\xi}$ whose orientation of chords has the properties in the statement of the Theorem. Note that each oriented chord $\overrightarrow{c}(j_\ell, k_\ell)$ determines a vector $\pm\underline{\dim}(X_{i(\ell)_1, i(\ell)_2})$. We say a chord $\overrightarrow{c}(j_\ell, k_\ell)$ of $\overrightarrow{d}$ is \textbf{positive} (resp. \textbf{negative}) if $j_\ell < k_\ell$ (resp. $k_\ell < j_\ell$). Let $C:= \left(\pm\underline{\dim}(X_{i(\ell)_{1}, i(\ell)_{2}})\right)_{\ell \in [n]} \in \mathbb{Z}^{n\times n}$ be the matrix whose row vectors are signed dimension vectors of representations corresponding to chords in $\overrightarrow{d}$. The sign of any row $\pm\underline{\dim}(X_{i(\ell)_1, i(\ell)_2})$ of $C$ is positive (resp. negative) if the corresponding oriented chord $\overrightarrow{c}(j_\ell, k_\ell)$ is positive (resp. negative) in $\overrightarrow{d}$. We say a representation $X_{i(\ell)_1, i(\ell)_2}$ is \textbf{positive} (resp. \textbf{negative}) if the sign of $\pm\underline{\dim}(X_{i(\ell)_1, i(\ell)_2})$ is positive (resp. negative).

To show that $\overline{\xi}$ is reachable, it is enough to show that $C \in \textbf{c}$-mat$(Q)$ as this will imply that $\overrightarrow{d} = \overrightarrow{d}_C$. By assumption, there are no subdiagrams of $\overrightarrow{d}$ of the form in the statement of the theorem. Thus, by Remark~\ref{homnonvan}, there is no pair of representations $X_{i(\ell)_1, i(\ell)_2}$ and $X_{i(\ell^\prime)_1, i(\ell^\prime)_2}$ such that $\pm\underline{\dim}(X_{i(\ell)_1, i(\ell)_2})$ and $\pm\underline{\dim}(X_{i(\ell^\prime)_1, i(\ell^\prime)_2})$ have the same sign and satisfy $\Hom_{kQ}(X_{i(\ell)_1, i(\ell)_2}, X_{i(\ell^\prime)_1, i(\ell^\prime)_2}) \neq 0$ or $\Hom_{kQ}(X_{i(\ell^\prime)_1, i(\ell^\prime)_2}, X_{i(\ell)_1, i(\ell)_2}) \neq 0.$ 

Next, we show that there exists a CES $\xi_C$ using exactly the representations $\{X_{i(\ell)_1, i(\ell)_2}\}_{\ell \in [n]}$ so that all of the negative representations appear before all of the positive representations in $\xi_C$. The result will then follow from Theorem~\ref{st}. To show that such a CES $\xi_C$ exists, Theorem~\ref{firstmainresult} implies it is enough to define a good chord labeling of $|\overrightarrow{d}|$ (see Remark~\ref{natmap}) with the property that any positive chord of $\overrightarrow{d}$ has a smaller chord label in $|\overrightarrow{d}|$ than the label of any negative chord in $|\overrightarrow{d}|$. If $\overrightarrow{d}$ has either no positive chords or no negative chords, then clearly any good chord labeling of $|\overrightarrow{d}|$ has the desired property. Thus we can assume that $\overrightarrow{d}$ has both positive and negative chords.

Suppose $\overrightarrow{d}$ has both positive and negative chords. Clearly, $\overrightarrow{d} = \text{pos}(\overrightarrow{d}) \sqcup \text{neg}(\overrightarrow{d})$ where $\text{pos}(\overrightarrow{d})$ (resp. $\text{neg}(\overrightarrow{d})$) denotes the set of positive (resp. negative) chords of $\overrightarrow{d}$. To define the desired good labeling it is enough to show how to label the positive chords. Then any good labeling of the remaining negative chords will produce an element $|\overrightarrow{d}|(n) \in \mathcal{D}(n)$ whose underlying diagram is $|\overrightarrow{d}|$. 

We proceed by induction on the size of $\text{pos}(\overrightarrow{d})$. If $\#\text{pos}(\overrightarrow{d}) = 1$, then denote the unique positive chord of $\overrightarrow{d}$ by $\overrightarrow{c}(j(1),k(1))$. By uniqueness and the fact that $\overrightarrow{d}$ is weakly separated at $j(1)$ and $k(1)$, the chord $\overrightarrow{c}(j(1),k(1)) \in \overrightarrow{d}$ is counterclockwise from all other chords connected to $j(1)$ and all other chords connected to $k(1)$. Label $c(j(1),k(1)) \in |\overrightarrow{d}|$ by 1.

Now suppose that $\#\text{pos}(\overrightarrow{d}) = k$ and that a desired labeling of the positive chords exists for any oriented diagrams with fewer positive chords. Suppose there exists $\overrightarrow{c}(j,k) \in \text{pos}(\overrightarrow{d})$ such that $\deg(j) = 1$ where $$\deg(j):= \#\{\text{chords connected to $j$ in $\overrightarrow{d}$}\}.$$ In this case, label $c(j,k) \in |\overrightarrow{d}|$ by 1. Now suppose there does not exist $\overrightarrow{c}(j,k) \in \text{pos}(\overrightarrow{d})$ where $\deg(j) = 1$. Then any chord $\overrightarrow{c}(j,k) \in \text{pos}(\overrightarrow{d})$ has $\deg(j) \ge 2.$

We claim that there exists $\overrightarrow{c}(j,k) \in \text{pos}(\overrightarrow{d})$ that is counterclockwise from all chords connected to $j$ and $k$. Suppose that no such oriented chord exists. Then any $\overrightarrow{c}(j,k) \in \text{pos}(\overrightarrow{d})$ is clockwise from an oriented chord connected to $j$ or one connected to $k$. Any oriented chord of $\overrightarrow{d}$ connected to $k$ that is counterclockwise from $\overrightarrow{c}(j,k)$ must be positive by the weakly separated property. The existence of such a chord means that $\overrightarrow{d}$ has a subdiagram of one of the forbidden forms. Thus no such oriented chord of $\overrightarrow{d}$ connected to $k$ exists.

Now let $\overrightarrow{c}(j,k) \in \text{pos}(\overrightarrow{d})$. If $\overrightarrow{c}(j,k)$ is counterclockwise from any chord connected to $j$, label $c(j,k) \in |\overrightarrow{d}|$ by 1.  If not, then $\overrightarrow{c}(j,k)$ is clockwise from an oriented chord of $\overrightarrow{d}$ that is connected to $j$. By the weakly separated property, the chord must be positive. Since $\overrightarrow{d}$ has no subdiagrams of the form appearing in the statement of the theorem, such an oriented chord must be of the form $\overrightarrow{c}(k_1,j)$. If $\overrightarrow{c}(k_1,j)$ is counterclockwise from any chord connected to $k_1$, label $c(k_1,j) \in |\overrightarrow{d}|$ by 1. Since there are only finitely many oriented chords of $\overrightarrow{d}$, this process will terminate. Thus we can find $\overrightarrow{c}(j,k) \in \overrightarrow{d}$ that is counterclockwise from all oriented chords of $\overrightarrow{d}$ connected to $j$ and $k$. Label $c(j,k) \in |\overrightarrow{d}|$ by 1.

Thus by induction we obtain a labeling of the positive chords with the desired property. The result follows.\end{proof}

The corollary below gives a complete classification of $\bc$-matrices in terms of oriented chord diagrams. This shows exactly which oriented diagrams are reachable via diagram mutation.

\begin{corollary}\label{c-matClassif}
The injective image of the map $\textbf{c}\text{-mat}(Q)\to \overrightarrow{\mathcal{D}}$ defined in Lemma \ref{cmatdiag} consists of exactly the oriented diagrams that are weakly separated at each marked point and which contain none of the configurations of Theorem \ref{reachcec} as subdiagrams.
\end{corollary}

\begin{remark}
Our work is related to that of Igusa and Ostroff (see \cite{io13}) where they develop the basic cluster theory using mixed cobinary trees.
\end{remark}

\section{Exceptional sequences and linear extensions}
\label{sec:pos}
In this section, we consider the problem of counting the number of {CESs} arising from a given CEC. It turns out this problem can be restated as the problem of counting the number of linear extensions of certain posets.

Let $n+1$ be the number of marked points on the boundary of the disk. Let $c(i,j)$ be a chord inside the disk. We define a clockwise (resp. counterclockwise) \textbf{rotation} of $c(i,j)$ about $i$ to be $c(i,j-1)$ (resp. $c(i,j+1)$) where $j-1$ (resp. $j+1$) denotes the congruence class of $j-1$ (resp $j+1$) mod $n+1$. 

\begin{notation}
Let $c(i,j)$ be a chord in a disk with $n+1$ marked points. Define $\tau c(i,j) := c(i-1, j-1)$ and $\tau^{-1} c(i,j) := c(i+1, j+1)$ where we consider $i\pm 1$ and $j\pm 1$ mod $n+1$. We also define $\rho c(i,j) := c(i,j-1)$ and $\rho^{-1} c(i,j) :=  c(i,j+1)$ where we consider $j\pm 1$ mod $n+1$.
\end{notation}

\begin{remark}
The maps $\tau$ and $\tau^{-1}$ agree with the Auslander-Reiten translation and its inverse, respectively. The maps $\rho$ and $\rho^{-1}$ encode the notion of clockwise and counterclockwise rotation of chords, respectively.
\end{remark}

Let $d \in \mathcal{D}$. Define the \textbf{poset} of $d$, denoted $\mathcal{P}_d$, to be the partially ordered set whose underlying set is the set of chords of $d$ with the relations $x \le y$ if $x$ and $y$ intersect at a marked point $i \in [0,n]$ and $y$ is the first chord obtained from $x$ by a sequence of clockwise {rotations} of $x$ about $i$. Note that $c(i,j_1) \in \mathcal{P}_d$ is \textbf{covered by} $c(i,j_2) \in \mathcal{P}_d$ (equivalently, $c(i,j_2)$ \textbf{covers} $c(i,j_1)$) if and only $c(i,j_2)$ is obtained by single clockwise rotation of $c(i,j_1)$ about $i$. We write $c(i,j_1) \lessdot c(i,j_2)$ to indicate that $c(i,j_2)$ covers $c(i,j_1).$

This construction defines a poset because any oriented cycle in the Hasse diagram of $\mathcal{P}_d$ arises from a cycle in $d$. Since $d$ is a tree, the diagram has no cycles. In Figure~\ref{fig:PosetExample}, we show a diagram $d \in \mathcal{D}$ and its poset $\mathcal{P}_d.$

\begin{figure}[ht]
\centering
\begin{minipage}[b]{0.4\textwidth}
\includegraphics[width=\textwidth]{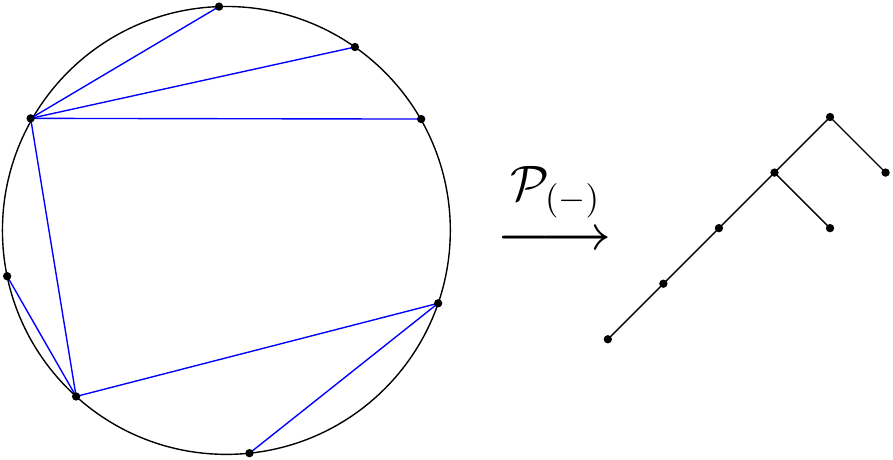}
\caption{A diagram and its poset.}
\label{fig:PosetExample}
\end{minipage}
\quad
\begin{minipage}[b]{0.4\textwidth}
\includegraphics[width=\textwidth]{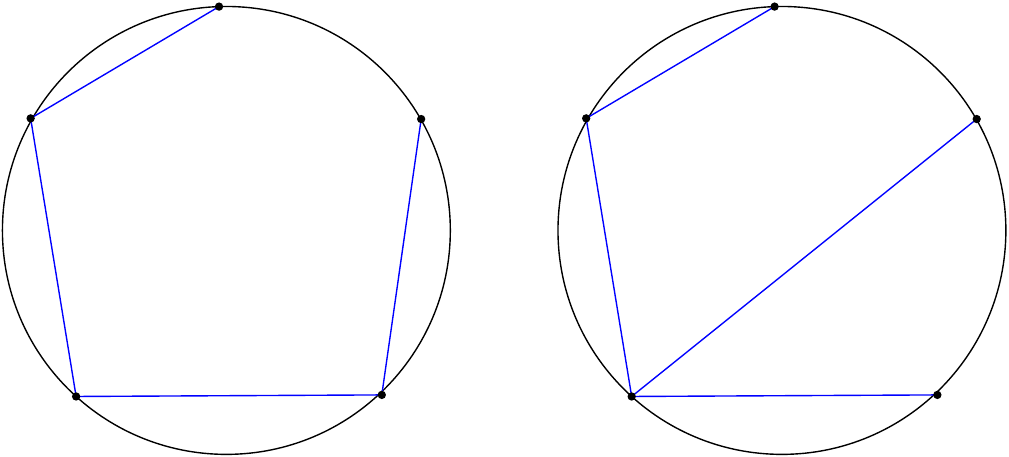}
\caption{Two diagrams with the same poset.}
\label{fig:SamePosets}
\end{minipage}
\end{figure}

In general, the map $\mathcal{D} \rightarrow \mathscr{P}(\mathcal{D}):=\{\mathcal{P}_d: d \in \mathcal{D}\}$ is not injective. For instance, each of the two diagrams in Figure~\ref{fig:SamePosets} have $\mathcal{P}_d = \textbf{4}$ where $\textbf{4}$ denotes the linearly-ordered poset with 4 elements. Our next result describes the posets arising from diagrams in $\mathcal{D}$.

\begin{theorem}\label{posetclassif}
A poset $\mathcal{P} \in \mathscr{P}(\mathcal{D})$ if and only if 

$\begin{array}{rllc}
i) & \text{each $x \in \mathcal{P}$ has at most two covers and covers at most two elements}\\
ii) & \text{the underlying graph of the Hasse diagram of $\mathcal{P}$ has no cycles}\\
iii) & \text{the Hasse diagram of $\mathcal{P}$ is connected}.\\
\end{array}$
\end{theorem}
\begin{proof}
Let $\mathcal{P}_d \in \mathscr{P}(\mathcal{D})$. By definition, $\mathcal{P}_d$ satisfies $i)$. It is also clear that the Hasse diagram of $\mathcal{P}_d$ is connected since $d$ is a connected graph. To see that $\mathcal{P}_d$ satisfies $ii)$, observe that if $\mathcal{P}_d$ had a cycle $C$ in its underlying graph, then $C$ must have an acyclic orientation. Note that we can assume that $C$ is a \textbf{minimal cycle} (i.e. the underlying graph of $C$ does not contain a proper subgraph that is a cycle). Thus there exists $x_C \in \mathcal{P}_d$ such that $x_C \in C$ is covered by two distinct elements $y, z \in C$. Observe that the cycle $C$ in the underlying graph of $\mathcal{P}_d$ is equivalent to a sequence of chords $\{c_i\}_{i = 0}^\ell$ of $d$ in which $y$ and $z$ appear exactly once and where for all $i \in [0,\ell]$ $c_i$ and $c_{i+1}$ (we consider the indices mod $\ell + 1$) share a marked point $j$ and no chord adjacent to $j$ appears between $c_i$ and $c_{i+1}$. Since the chords of $d$ are noncrossing, such a sequence cannot exist. Thus $\mathcal{P}_d$ has no cycles.

To prove the converse, we proceed by induction on the number of elements of $\mathcal{P}$ where $\mathcal{P}$ is a poset satisfying conditions $i), ii), iii)$. If $\#\mathcal{P} = 1$, then $\mathcal{P}$ is the unique poset with one element and $\mathcal{P} = \mathcal{P}_d$ where $d$ is the unique chord diagram associated to the disk with two marked points that is a spanning tree. Assume that for any poset $\mathcal{P}$ satisfying conditions $i), ii), iii)$ with $\#\mathcal{P} = n$ for some positive integer $n$ there exists a chord diagram $d$ such that $\mathcal{P} = \mathcal{P}_d$. Let $\mathcal{Q}$ be a poset satisfying the above conditions and assume $\#\mathcal{Q} = n+1.$ Let $x \in \mathcal{Q}$ be a maximal element. 

If the poset $\mathcal{Q}-\{x\}$ has a disconnected Hasse diagram, then $\mathcal{Q} - \{x\} = \mathcal{Q}_1 + \mathcal{Q}_2$ where $\mathcal{Q}_i$ satisfies $i), ii), iii)$ for $i \in [2]$. By induction, there exists positive integers $k_1, k_2$ satisfying $k_1 + k_2 = n$ and diagrams $$d_i \in \mathcal{D}_{k_i} := \{\text{diagrams } \{c_i(i_\ell, j_\ell)\}_{\ell\in [k_i]} \text{ in a disk with $k_i+1$ marked points}\}$$ where $\mathcal{Q}_i = \mathcal{P}_{d_i}$ for $i \in [2].$ Define $d_1\sqcup d_2 := \{c^\prime(i^\prime_\ell,j^\prime_\ell)\}_{\ell \in [n]}$ to be the diagram in the disk with $n+2$ marked points as follows (we give an example of this operation below with $k_1 = 3$ and $k_2 = 2$ so that $n = k_1 + k_2 = 5$)
{$$\begin{array}{rclcc}
c^\prime(i^\prime_\ell, j^\prime_\ell) & := & \left\{\begin{array}{lcl} c_1(i_\ell, j_\ell) & : & \text{if $\ell \in [k_1]$}\\ \tau^{-(k_1+1)}c_2(i_{\ell - k_1},j_{\ell - k_1})  & : & \text{if $\ell \in [k_1 + 1, n]$}. \end{array}\right.
\end{array}$$}

$$\begin{array}{ccccc}
\includegraphics[scale=.85]{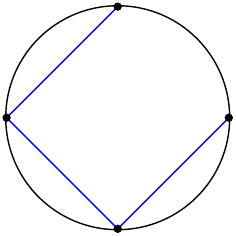} \raisebox{.35in}{$\text{ $\Large\sqcup$ }$} \includegraphics[scale=.85]{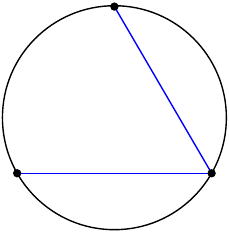} \raisebox{.35in}{ \text{ $\Large=$ } } \includegraphics[scale=.85]{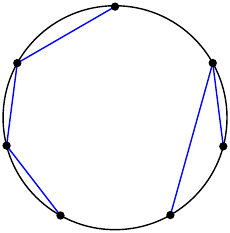}
\end{array}$$

\noindent Define $c^\prime(i^\prime_{n+1}, j^\prime_{n+1}): = c(k_1, n+1)$ and then $d:= \{c^\prime(i^\prime_\ell, j^\prime_\ell)\}_{\ell \in [n+1]}$ satisfies $i), ii), iii),$ and $\mathcal{Q} = \mathcal{P}_d.$

If the Hasse diagram of $\mathcal{Q}-\{x\}$ is connected, then by induction the poset $\mathcal{Q} - \{x\} = \mathcal{P}_d$ for some diagram $d = \{c(i_\ell, j_\ell)\}_{\ell \in [n]} \in \mathcal{D}_n$ where we assume $i_\ell < j_\ell$. Since the Hasse diagram of $\mathcal{Q}-\{x\}$ is connected, it follows that $x$ covers a unique element in $\mathcal{Q}$. Let $y  = c(i(y),j(y)) \in \mathcal{Q} - \{x\}$ ($i(y) < j(y)$) denote the unique element that is covered by $x$ in $\mathcal{Q}$. This means that either $y = c(i(y),i(y) + 1)$ or there are no chords in $d$ obtained by a clockwise rotation of $c(i(y), j(y))$ about $j(y)$. We assume $y = c(i(y), i(y)+1)$. The proof is similar in the case where there are no chords in $d$ obtained by a clockwise rotation of $c(i(y), j(y))$ about $j(y)$.

Regard $d$ as an element of $\mathcal{D}_{n+1}$ by replacing it with $\widetilde{d}:= \{c^\prime(i^\prime_\ell, j^\prime_\ell)\}_{\ell \in [n]}$ defined by (we give an example of this operation below with $n = 6$)

{$$\begin{array}{rclcc}
c^\prime(i^\prime_\ell, j^\prime_\ell) & := & \left\{\begin{array}{lcl} \rho^{-1}c(i_\ell, j_\ell) & : & \text{if $i_\ell \le i(y)$ and $j(y) \le j_\ell$,}\\  \tau^{-1}c(i_{\ell},j_{\ell})  & : & \text{if $j(y)  \le i_\ell$,}\\
c(i_\ell, j_\ell) & : & \text{otherwise}. \end{array}\right.
\end{array}$$}
$$\begin{array}{ccccccccc}
\raisebox{.35in}{$d$} & \raisebox{.35in}{$=$} & \includegraphics[scale=.85]{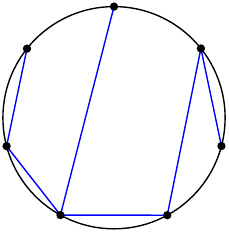} \raisebox{.35in}{ \text{ $\Large \longrightarrow$ } } & \raisebox{.35in}{$\widetilde{d}$} & \raisebox{.35in}{$=$} & \includegraphics[scale=.85]{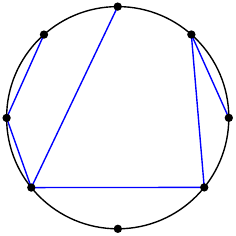}
\end{array}$$
\noindent Define $c^\prime(i^\prime_{n+1}, j^\prime_{n+1}) := c(i(y), i(y)+1)$ and put $d^\prime := \{c^\prime(i^\prime_\ell, j^\prime_\ell)\}_{\ell \in [n+1]}$. As $\mathcal{Q}-\{x\}$ satisfies $i), ii),$ and $iii)$, {it is clear that the resulting chord diagram $d^\prime$ satisfies $\mathcal{P} = \mathcal{P}_{d^\prime}$.} 
\end{proof}

Let $\mathcal{P}$ be a finite poset with $m = \#\mathcal{P}$. Let $f: \mathcal{P} \to \textbf{m}$ be an injective, order-preserving map (i.e. $x \le y$ implies $f(x) \le f(y)$ for all $x,y \in \mathcal{P}$) where $\textbf{m}$ is the linearly-ordered poset with $m$ elements. We call $f$ a \textbf{linear extension} of $\mathcal{P}$. We denote the set of linear extensions of $\mathcal{P}$ by $\mathscr{L}(\mathcal{P})$. Note that since $f$ is an injective map between sets of the same cardinality, $f$ is a bijective map between those sets.

\begin{theorem}\label{CESsandlinexts}
Let $d = \{c(i_\ell, j_\ell)\}_{\ell\in [n]} \in \mathcal{D}$ and let $\overline{\xi}$ denote the corresponding complete exceptional collection. Let $\text{CES}(\overline{\xi})$ denote the set of CESs that can be formed using only the representations appearing in $\overline{\xi}$. Then the map $\chi: \text{CES}(\overline{\xi}) \to \mathscr{L}(\mathcal{P}_d)$ defined by $(X_{i_1,j_1},\ldots, X_{i_n,j_n}) \stackrel{\chi_2}{\longmapsto} \{(c(i_\ell,j_\ell), n+1-\ell)\}_{\ell \in [n]} \stackrel{\chi_1}{\longmapsto} (f(c(i_\ell,j_\ell)) := n+1-\ell)$ is a bijection.
\end{theorem}
\begin{proof}
The map $\chi_2 = \widetilde{\Phi}: \text{CES}(\overline{\xi}) \to \mathcal{D}(n)$ is a bijection by Theorem~\ref{firstmainresult}. Thus it is enough to prove that $\chi_1: \mathcal{D}(n) \to \mathscr{L}(\mathcal{P}_d)$ is a bijection.

First, we show that $\chi_1(d(n)) \in \mathscr{L}(\mathcal{P}_d)$ for any $d(n) \in \mathcal{D}(n)$. Let $d(n) \in \mathcal{D}(n)$ and let $f := \chi_1(d(n))$. Since the chord-labeling of $d(n)$ is good, if $(c_1, \ell_1), (c_2, \ell_2)$ are two labeled chords of $d(n)$ satisfying $c_1 \le c_2$, then $f(c_1) = \ell_1 \le \ell_2 = f(c_2).$ Thus $f$ is order-preserving. As the chords of $d(n)$ are bijectively labeled by $[n]$, we have that $f$ is bijective so $f \in \mathscr{L}(\mathcal{P}_d)$.

Next, define a map 
$$\begin{array}{rcl}
\mathscr{L}(\mathcal{P}_d) & \stackrel{\varphi}{\longrightarrow} & \mathcal{D}(n)\\
f & \longmapsto & \{(c(i_\ell, j_\ell), f(c(i_\ell, j_\ell)))\}_{\ell \in [n]}.
\end{array}$$
To see that $\varphi(f) \in \mathcal{D}(n)$ for any $f \in \mathscr{L}(\mathcal{P}_d)$, consider two labeled chords $(c_1, f(c_1))$ and $(c_2, f(c_2))$ belonging to $\varphi(f)$ where $c_1 \le c_2$. Since $f$ is order-preserving, $f(c_1) \le f(c_2).$ Thus the chord-labeling of $\varphi(f)$ is good so $\varphi(f) \in \mathscr{L}(\mathcal{P}_d)$.

Lastly, we have that $$\chi_1(\varphi(f)) = \chi_1(\{(c(i_\ell, j_\ell), f(c(i_\ell, j_\ell)))\}_{\ell \in [n]}) = f$$ and $$\varphi(\chi_1(\{(c(i_\ell, j_\ell), \ell)\}_{\ell \in [n]})) = \varphi(f(c(i_\ell, j_\ell)):= \ell) = \{(c(i_\ell, j_\ell), \ell)\}_{\ell \in [n]}$$ so $\varphi = \chi^{-1}$. Thus $\chi_1$ is a bijection.
\end{proof}

\section{Permutations and exceptional sequences}
\label{sec: perm}

The authors' initial reason for starting this project was to investigate the problem below posed by D. Rupel and G. Todorov during the 2014 MRC program on cluster algebras. To state this problem, we momentarily allow $Q$ to be any acyclic quiver.

Let \textbf{c}$(n)$-mat($Q$) $:= \{C_R(n) : R \in ET(\widehat{Q})\}$. Let $C(n) \in \textbf{c}(n)$-mat($Q$) and let $\{\overrightarrow{c_i}\}_{i \in [n]}$ denote its \textbf{c}-vectors. Recall that it follows from the work of Chavez (see \cite{c12}) that any $\textbf{c}$-vector $\overrightarrow{c_i}$ of $Q$ satisfies $|\overrightarrow{c_i}| = \underline{\dim}(V_i)$ for some indecomposable representation $V_i$ of $Q$.

\begin{problem}\label{permprob}
Let $\widehat{Q}$ be the framed quiver of an acyclic quiver $Q$. Let $\mu_{i_k} \circ \cdots \mu_{i_1}(\widehat{Q}) \in ET(\widehat{Q})$ and let $C(n) \in \textbf{c}(n)$-mat($Q$) be the corresponding $\bc$-matrix with an ordering $(\overrightarrow{c_n},\ldots, \overrightarrow{c_1})$ on the \textbf{c}-vectors, where $\overrightarrow{c_i}$ denotes the $i$th row of $C(n)$. Let $V_i$ denote the indecomposable representation of $Q$ satisfying $|\overrightarrow{c_i}| = \underline{\dim}(V_i)$.  By Theorem~\ref{st}, there exists a permutation of the \textbf{c}-vectors of $C(n)$, denoted by $\sigma$, where $(V_{\sigma(n)}, \ldots, V_{\sigma(1)})$ is a CES with the property that if there exist positive \textbf{c}-vectors in $C(n)$, then there exists $k \in [n]$ such that $\overrightarrow{c_{\sigma(i)}}$ is positive if and only if $i \in [k].$ Describe these permutations.
\end{problem}

We use the results of this paper to provide a solution to Problem~\ref{permprob} where $Q$ is the linearly-ordered $\mathbb{A}_n$ quiver. Our solution goes beyond the statement of this problem in that we describe all permutations $\sigma$ for which $(V_{\sigma(n)}, \ldots, V_{\sigma(1)})$ is a CES.

\begin{remark}
It is hoped that these permutations can be described for any framed quiver $\widehat{Q}$ where $Q$ is an acyclic quiver, but the techniques of the current paper are only applicable for the linearly-ordered type $\AAA_n$ quiver.  Extending our results to other types will be the subject of future work.
\end{remark}

\subsection{Permutations via linear extensions}
Let $\cD[n]$ be the set of all chord diagrams with $n$ chords labeled bijectively by the elements of $[n]$ (not necessarily with a good labeling).  Recall that for every $\bc$-vector $\overrightarrow{c_\ell}$, we have $\overrightarrow{c_\ell} = \pm \underline{\dim}(X_{i_\ell,j_\ell})$ for some $0 \leq i_\ell < j_\ell \leq n$ (see \cite{c12}).  There is a natural map $\kappa:$ \textbf{c}$(n)$-mat($Q$) $\rightarrow \cD[n]$ defined by
\[
\{\overrightarrow{c_\ell}\}_{\ell \in[n]} \mapsto \{(c(i_\ell,j_\ell),\ell)\}_{\ell \in[n]}.
\]
For convenience, we set $c_\ell := c(i_\ell, j_\ell)$ and $V_\ell := X_{i_\ell,j_\ell}$.

Let $R := \mu_{i_k}\circ \cdots \circ \mu_{i_1}(\widehat{Q})$ and let $C(n):=C_R(n) \in$ \textbf{c}$(n)$-mat($Q$), and let $d_{C(n)}:=\kappa(C(n))$.  We write $\overline{d_{C(n)}}$ for the underlying unlabeled diagram.  As in Section \ref{sec:pos}, $\overline{d_{C(n)}}$ defines a poset $\mathcal{P}_{\overline{d_{C(n)}}}$.  One can write down all linear extensions $\{f_k\}_{k=1}^{p}$ of this poset, where we have set $p := \#\mathscr{L}(\mathcal{P}_{\overline{d_{C(n)}}})$.  Each $f_k$ will give a permutation $\varsigma_k \in \mathfrak{S}_n$
 in the following way.  By definition, $f_k$ defines a total ordering on the chords $\{c_\ell\}_{\ell\in[n]}$ which preserves the order coming from $\mathcal{P}_{\overline{d_{C(n)}}}$.  Define $\varsigma_k$ by $\ell \mapsto f_k(c_\ell)$.  This process gives a collection of permutations $\{\varsigma_k\}_{k=1}^{p}$.  Now, set $\sigma_k := \varsigma_k^{-1}$.
 \begin{theorem}\label{permproof}
Let $\{\sigma_k\}_{k=1}^p$ be defined as above.  For every $k \in [p]$, the ordered list of representations \\ $(V_{\sigma_k(n)},\ldots,V_{\sigma_k(1)})$ is a CES.
 \end{theorem}
 
Before proving the theorem, we present an example illustrating the method for finding the permutations $\{\sigma_k\}_{k=1}^p$.

\begin{example}\label{example2}
Let $\widehat{Q}$ be the framed linearly-ordered type $\AAA_4$ quiver.  Perform the mutation sequence $\mu_3 \circ \mu_2$.  We obtain the following $\bc$-matrix, labeled diagram, and labeled poset.
\[\raisebox{.2in}{$\left[\begin{array}{c c c c}
1 & 0 & 0 & 0 \\
0 & 0 & 1 & 0 \\
0 & -1 & -1 & 0 \\
0 & 1 & 1 & 1
\end{array}\right]$} \hspace{1cm} \raisebox{-.285in}{$\includegraphics[scale=.6]{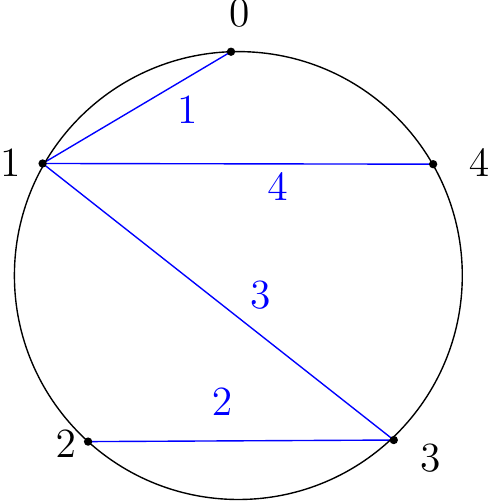}$} \hspace{1cm} \raisebox{-.075in}{\includegraphics[scale=1]{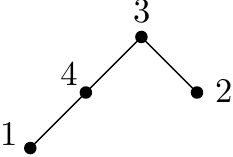}} 
\]
We list the linear extensions of the underlying poset and give, for each linear extension, the permutation gotten by mapping the label on each element of the poset above to its label in the posets below.

\[\includegraphics[scale=1]{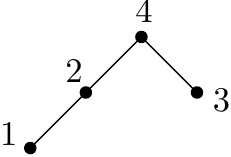} \hspace{1cm} \includegraphics[scale=1]{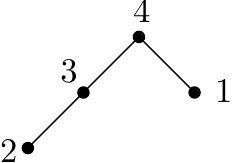} \hspace{1cm} \includegraphics[scale=1]{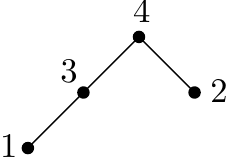}\]
\[\varsigma_1 = (423)  \hspace{1.5cm} \varsigma_2 = (12)(43) \hspace{1.6cm} \varsigma_3 = (43)\]
Taking inverses, we have $\sigma_1 = (324),\ \sigma_2 = (12)(43),$ and $\sigma_3 = (43)$.  Thus, we have the three CESs 
\[\ \ (V_3,V_2,V_4,V_1) \hspace{1.1cm} (V_3,V_4,V_1,V_2) \hspace{1.2cm} (V_3,V_4,V_2,V_1).\]
In other words, we have the CESs $(X_{1,3},X_{2,3},X_{1,4},X_{0,1}),\ (X_{1,3},X_{1,4},X_{0,1},X_{2,3}),$ and $(X_{1,3},X_{1,4},X_{2,3},X_{0,1})$.
\end{example}

\begin{proof}[Proof of Theorem~\ref{permproof}]
By the definition of $\varsigma_k$, we have $\{(c_\ell,\varsigma_k(\ell))\}_{\ell \in [n]} = \{(c_\ell,f_k(c_\ell))\}_{\ell \in [n]}$.  Recall the map 
$$\begin{array}{rcl}
\mathscr{L}(\mathcal{P}_d) & \stackrel{\varphi}{\longrightarrow} & \mathcal{D}(n)\\
f & \longmapsto & \{(c_\ell, f(c_\ell))\}_{\ell \in [n]}
\end{array}$$
defined in the proof of Theorem \ref{CESsandlinexts}.
Notice that $\varphi(f_k) = \{(c_\ell,f_k(c_\ell))\}_{\ell \in[n]}$; thus, as $\{(c_\ell,f_k(c_\ell))\}_{\ell\in[n]}$ is in the image of $\varphi$, it is a diagram with a good labeling.

Now consider the map
$$\begin{array}{rcl}
\mathcal{D}(n) & \stackrel{\widetilde{\Psi}}{\longrightarrow} & \mathcal{E}(n)\\
\{(c_\ell, \ell)\}_{\ell \in [n]} & \longmapsto & (V_n, \ldots, V_1)
\end{array}$$ 
defined in the proof of Theorem \ref{firstmainresult}.  Notice that, $\widetilde{\Psi}(\{(c_\ell,\varsigma_k(\ell))\}_{\ell \in [n]})=(V_{\varsigma_k^{-1}(n)},\ldots,V_{\varsigma_k^{-1}(1)})=(V_{\sigma_k(n)},\ldots,V_{\sigma_k(1)})$ since $\varsigma_k^{-1} = \sigma_k$.  Being in the image of $\widetilde{\Psi}$ guarantees that $(V_{\sigma_k(n)},\ldots,V_{\sigma_k(1)})$ is a CES, as desired.
\end{proof}

\section{Applications}
\label{sec:app}

Here we showcase some interesting results that follow easily from our main theorems.

\subsection{Labeled trees}
In \cite[p. 67]{sw86}, Stanton and White gave a nonpositive formula for the number of vertex-labeled trees with a fixed number of leaves. By connecting our work with that of Goulden and Yong \cite{gy02}, we obtain a positive expression for this number. 
\begin{theorem}\label{trees}Let $T_{n+1}(r) := \{\text{trees on } [n+1] \text{ with } r \text{ leaves}\}$ and $\mathcal{D} := \{\text{diagrams } d = \{c(i_\ell,j_\ell)\}_{\ell \in [n]}\}$. Then
$$\begin{array}{rcl}
\#T_{n+1}(r) & = &\displaystyle \sum_{\begin{array}{c}\small\text{$d \in \mathcal{D}:$ \ $d \text{ has $r$ chords }c(i_j,i_j+1)$}\end{array}}\#\mathscr{L}(\mathcal{P}_d).
\end{array}$$ 
\end{theorem}
\begin{proof}
Observe that 
$$\begin{array}{rcl}
\displaystyle \sum_{\begin{array}{c}\small\text{$d \in \mathcal{D}$ : $d$ \text{ has $r$}}\\ \small \text{chords $c(i_j,i_j+1)$}\end{array}}\#\mathscr{L}(\mathcal{P}_d) & = & \displaystyle \sum_{\begin{array}{c}\small\text{$d \in \mathcal{D}: d \text{ has $r$}$}\\ \small \text{$\text{chords $c(i_j, i_j + 1)$}$}\end{array}}\#\{\text{good labelings of $d$}\} \\ 
& = & \#\left\{d(n) \in \mathcal{D}(n): \begin{array}{l}d(n) \text{ has $r$ chords $c(i_j,i_j + 1)$}\\ \text{for some $i_1,\ldots, i_r \in [0,n]$}\end{array}\right\}
\end{array}$$
where we consider $i_j + 1$ mod $n+1$. By \cite[Theorem 1.1]{gy02}, we have a bijection between diagrams $d \in \mathcal{D}$ with $r$ chords of the form $c(i_j, i_j + 1)$ for some $i_1, \ldots, i_r  \in [0,n]$ with good labelings and elements of $T_{n+1}(r)$. 
\end{proof}
\begin{corollary}
We have $(n+1)^{n-1} = \sum_{d \in \mathcal{D}} \#\mathscr{L}(\mathcal{P}_d)$.
\end{corollary}
\begin{proof}
Let $T_{n+1} := \{\text{trees on [n+1]}\}.$ One has that
$$\begin{array}{cclccc}
(n+1)^{n-1} & = & \#T_{n+1} \\
& = & \displaystyle\sum_{r \ge 0} \#T_{n+1}(r) \\
& = & \displaystyle\sum_{r \ge 0}  \sum_{\begin{array}{c}\small\text{$d \in \mathcal{D}$ : $d$ \text{ has $r$}}\\ \small \text{chords $c(i_j,i_j+1)$}\end{array}}\#\mathscr{L}(\mathcal{P}_d) & \text{(by Theorem~\ref{trees})} \\
& = & \displaystyle \sum_{d \in \mathcal{D}} \#\mathscr{L}(\mathcal{P}_d).
\end{array}$$\end{proof}

\begin{remark}
It is natural to consider the connection between our work and parking functions of type $\AAA$.  This connection is investigated in \cite{gg13}.
\end{remark}

\subsection{Reddening sequences}
In \cite{k12}, Keller proves that for any quiver $Q$, any two reddening mutation sequences applied to $\widehat{Q}$ produce isomorphic ice quivers.  As mentioned in \cite{kel13}, his proof is highly dependent on representation theory and geometry, but the statement is purely combinatorial--we give a combinatorial proof of this result for the linearly-ordered quiver $Q$.

Let $R \in EG(\widehat{Q})$.  A mutable vertex $i$ of $\widehat{Q}$ is called \textbf{green} in $R$ if there are no arrows $j \to i$ with $j \in [n+1,m]$.  Otherwise, $i$ is called \textbf{red}.  A sequence of mutations $\mu_{i_r}\circ \cdots \circ \mu_{i_1}$ is \textbf{reddening} if all vertices of the mutated quiver $\mu_{i_r}\circ \cdots \circ \mu_{i_1}(\widehat{Q})$ are red.  Recall that an isomorphism of quivers that fixes the frozen vertices is called a \textbf{frozen isomorphism}.  We now state the theorem.
\begin{theorem}
If $\mu_{i_r}\circ \cdots \circ \mu_{i_1}$ and $\mu_{j_s}\circ \cdots \circ \mu_{j_1}$ are two reddening sequences of $\widehat{Q}$, then there is a frozen isomorphism $\mu_{i_r}\circ \cdots \circ \mu_{i_1}(\widehat{Q}) \cong \mu_{j_s}\circ \cdots \circ \mu_{j_1}(\widehat{Q})$.
\end{theorem}
\begin{proof}
Let $\mu_{i_r}\circ \cdots \circ \mu_{i_1}$ be any reddening sequence.  Denote by $C$ the $\bc$-matrix of the final quiver.  By Corollary~\ref{c-matClassif}, $C$ corresponds to an oriented diagram  $\overrightarrow{d}_C$ with all chords of the form $\overrightarrow{c}(j,i)$ for some $i$ and $j$ satisfying $i < j$.  As $\overrightarrow{d}_C$ avoids the configurations of Theorem \ref{reachcec}, we conclude that $\overrightarrow{d}_C = \{\overrightarrow{c}(i, i-1)\}_{i\in[n]}$ and $C = -I_n$.  Since \textbf{c}-matrices are in bijection with ice quivers in $EG(\widehat{Q})$ (see \cite{by14}) and since $\widecheck{Q}$ is an ice quiver in $EG(\widehat{Q})$ whose \textbf{c}-matrix is $-I_n$, we obtain the desired result.
\end{proof}

\subsection{Noncrossing partitions and exceptional sequences}
A \textbf{partition} of $[n]$ is a collection $\pi = \{B_\alpha\}_{\alpha \in I} \in 2^{[n]}$ of subsets of $[n]$ called \textbf{blocks} that are nonempty, pairwise disjoint, and whose union is $[n].$ We denote the lattice of set partitions of $[n]$ by $\Pi_n$. A set partition $\pi = \{B_{\alpha}\}_{\alpha \in I} \in \Pi_n$ is called \textbf{noncrossing} if for any $i < j < k < \ell$ where $i, k \in B_{\alpha_1}$ and $j, \ell \in B_{\alpha_2}$, one has $B_{\alpha_1} = B_{\alpha_2}.$ We denote the lattice of noncrossing partitions of $[n]$ by  $NC^{\mathbb{A}}(n)$.

Label the vertices of a convex $n$-gon $\mathcal{S}$ with elements of $[n]$ so that reading the vertices of $\mathcal{S}$ counterclockwise determines an increasing sequence mod $n$. We can thus regard $\pi = \{B_\alpha\}_{\alpha \in I} \in NC^\mathbb{A}(n)$ as a collection of convex hulls $B_\alpha$ of vertices of $\mathcal{S}$ where $B_\alpha$ has empty intersection with any other block $B_{\alpha^\prime}$.

Let $n = 5$. The following partitions all belong to $\Pi_5$, but only $\pi_1, \pi_2, \pi_3 \in NC^\mathbb{A}(5).$ $$\pi_1 = \{\{1\}, \{2,4,5\}, \{3\}\}, \pi_2 = \{\{1,4\}, \{2,3\}, \{5\}\}, \pi_3 = \{\{1,2,3\}, \{4,5\}\}, \pi_4 = \{\{1,3,4\}, \{2,5\}\}$$  Below we represent the partitions $\pi_1, \ldots, \pi_4$ as convex hulls of sets of vertices of a convex pentagon.  We see from this representation that $\pi_4 \not \in NC^\mathbb{A}(5).$

\begin{center}
\includegraphics[scale=.75]{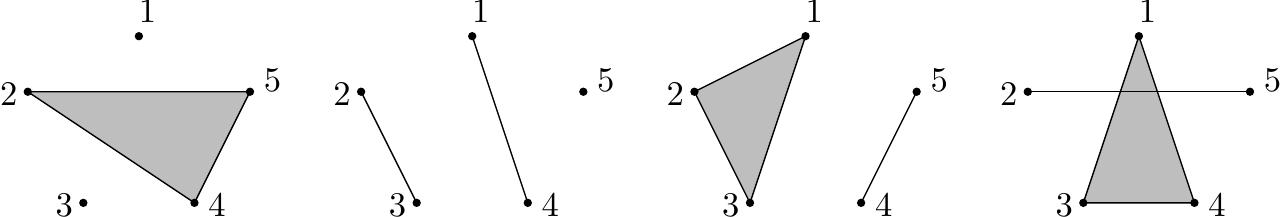}
\end{center}

\begin{theorem}\label{chains}
Let $k \in [n].$ There is a bijection between $\mathcal{D}(k)$ and the following chains in $NC^{\AAA}(n+1)$ \[\left\{\{\{i\}\}_{i \in [n+1]}  < \pi_1 < \cdots < \pi_{k} \in (NC^\mathbb{A}(n+1))^{k+1}: \begin{array}{c}\pi_j = (\pi_{j-1} \backslash\{B_{\alpha}, B_\beta\})\sqcup \{B_\alpha \sqcup B_\beta\}\\ \text{for some $B_\alpha \neq B_\beta$ in $\pi_{j-1}$}\end{array}\right\}.\] In particular, when $k = n$, there is a bijection between $\mathcal{D}(k)$ and maximal chains in $NC^{\mathbb{A}}(n+1)$.
\end{theorem}
\begin{proof}
Let $d(k) = \{(c(i_\ell,j_\ell),\ell)\}_{\ell \in [k]} \in \mathcal{D}(k)$. Define $\pi_{d(k),1} := \{\{i\}\}_{i \in [n+1]} \in \Pi_{n+1}.$ Next, define $\pi_{d(k),2} := \left(\pi_{d(k),1}\backslash \{\{i_1 + 1\}, \{j_1+1\}\}\right) \sqcup \{i_1+1,j_1+1\}$. Now assume that $\pi_{d(k),s}$ has been defined for some $s \in [k]$. Define $\pi_{d(k),s+1}$ to be the partition obtained by merging the blocks of $\pi_{d(k),s}$ containing $i_{s}+1$ and $j_{s}+1$. Now define $f(d(k)) := \{\pi_{d(k),s}: s \in [k+1]\}.$

It is clear that $f(d(k))$ is a chain in $\Pi_{n+1}$ with the desired property as $\pi_1 \lessdot \pi_2$ in $\Pi_{n+1}$ if and only if $\pi_2$ is obtained from $\pi_1$ by merging exactly two distinct blocks of $\pi_1.$ To see that each $\pi_{d(k),s} \in NC^\mathbb{A}(n+1)$ for all $s \in [k+1]$, suppose a crossing of two blocks occurs in a partition appearing in $f(d(k))$. Let $\pi_{d(k),s}$ be the smallest partition of $f(d(k))$ (with respect to the partial order on set partitions) with two blocks crossing blocks $B_1$ and $B_2$. Without loss of generality, we assume that $B_2 \in \pi_{d(k),s}$ is obtained by merging the blocks $B_{\alpha_1}, B_{\alpha_2} \in \pi_{d(k),s-1}$ containing $i_{s-1} + 1$ and $j_{s-1}+1$, respectively. This means that $d(k)$ has a chord $c(i_{s-1},j_{s-1})$ that crosses at least one other chord of $d(k)$. This contradicts that $d(k) \in \mathcal{D}(k)$. Thus $f(d(k))$ is a chain in $NC^\mathbb{A}(n+1)$ with the desired property.

Next, we define a map $g$ that is the inverse of $f$. Let $C = (\pi_1 = \{\{i\}\}_{i \in [n+1]} < \pi_2 < \cdots < \pi_{k+1}) \in (NC^\mathbb{A}(n+1))^{k+1}$ be a chain where each partition in $C$ satisfies $\pi_j = (\pi_{j-1}\backslash\{B_\alpha, B_\beta\}) \sqcup \{B_\alpha \sqcup B_\beta\}$ for some $B_\alpha \neq B_\beta$ in $\pi_{j-1}$. As $\pi_2 = \left(\pi_1 \backslash \{\{s_1\}, \{t_1\}\}\right)\sqcup \{s_1, t_1\}$, define $c(i_1,j_1) := c(s_1-1, t_1-1)$ where we consider $s_1 - 1$ and $t_1-1$ mod $n+1.$ Assume $s_1 < t_1$. If $t_1$ is in a block of size 3 in $\pi_3$, let $t$ denote the element of this block where $t \neq s_1, t_1.$ If $t$ satisfies $s_1 < t < t_1$, define $c(i_2,j_2) := c(s_1-1,t-1)$. Otherwise, define $c(i_2,j_2) := c(t_1-1, t-1)$. If there is no block of size 3 in $\pi_3$, define $c(i_2,j_2) := c(s_2-1, t_2-1)$ where $\{s_2\}$ and $\{t_2\}$ were singleton blocks in $\pi_2$ and $\{s_2,t_2\}$ is a block in $\pi_3.$ 

Now suppose we have defined $c(i_r, j_r)$. Let $B$ denote the block of $\pi_{r+2}$ obtained by merging two blocks of $\pi_{r+1}$. If $B$ is obtained by merging two singleton blocks $\{s_{r+1}\}, \{t_{r+1}\} \in \pi_{r+1}$, define $c(i_{r+1}, j_{r+1}) := c(s_{r+1}-1, t_{r+1}-1).$ Otherwise, $B = B_1\sqcup B_2$ where $B_1,B_2 \in \pi_{r+1}.$ Now note that, up to rotation and up to adding or deleting elements of $[n+1]$ for $B_1$ and $B_2$, $B_1\sqcup B_2$ appears in $\pi_{r+2}$ as follows. 

\begin{center}
\includegraphics[scale= 1.4]{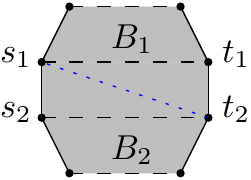}
\end{center}

\noindent Thus define $c(i_{r+1}, j_{r+1}) := c(s_1-1,t_2-1)$. Finally, put $g(C) := \{(c(i_\ell,j_\ell),\ell): \ell \in [k]\}$.

We claim that $g(C)$ has no crossing chords. Suppose $(c(s_i,t_i), i)$ and $(c(s_j,t_j),j)$ are crossing chords in $g(C)$ with $i < j$ and $i,j \in [k]$. We further assume that $$j = \min\{j^\prime \in [i+1,k]: \text{$(c(s_{j^\prime},t_{j^\prime}), j^\prime)$ crosses $(c(s_i,t_i),i)$ in $g(C)$}\}.$$
We observe that $s_i + 1, t_i + 1 \in B_1$ for some block $B_1 \in \pi_{j}$ and that $s_j+1, t_j+1 \in B_2$ for some block $B_2 \in \pi_{j+1}$. We further observe that $s_j+1, t_j+1 \not \in B_1$ otherwise, by the definition of the map $g$, the chords $(c(s_i,t_i),i)$ and $(c(s_j,t_j),j)$ would be noncrossing. Thus $B_1, B_2 \in \pi_{j+1}$ are distinct blocks that cross so $\pi_{j+1} \not \in NC^\mathbb{A}(n+1).$ We conclude that $g(C)$ has no crossing chords so $g(C) \in \mathcal{D}(k)$.

To complete the proof, we show that $g\circ f = 1_{\mathcal{D}(k)}$. The proof that $f \circ g$ is the identity map is similar. Let $d(k) \in \mathcal{D}(k)$. Then $f(d(k)) = \{\{i\}\}_{i \in [n+1]} < \pi_1 < \cdots < \pi_k$ where for any $s \in [k]$ we have $$\pi_s = \left(\pi_{s-1}\backslash\{B_\alpha, B_\beta\}\right) \sqcup \{B_\alpha, B_\beta\}$$ where $i_{s-1} + 1 \in B_\alpha$ and $j_{s-1} + 1 \in B_\beta.$ Then we have
$$g(f(d(k)))  =  \{c((i_\ell +1) - 1, (j_\ell +1) - 1), \ell)\}_{\ell \in [k]} = \{(c(i_\ell, j_\ell), \ell)\}_{\ell \in [k]}.$$
\end{proof}

\begin{example}
Here we give examples of the bijection from the previous theorem with $k = 4$.

$$\begin{array}{ccc}
\includegraphics[scale=.7]{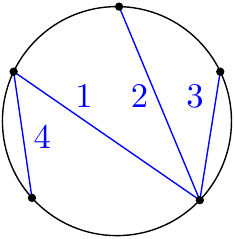} & \raisebox{.3in}{$\longmapsto$} & \raisebox{.1in}{\includegraphics[scale=.7]{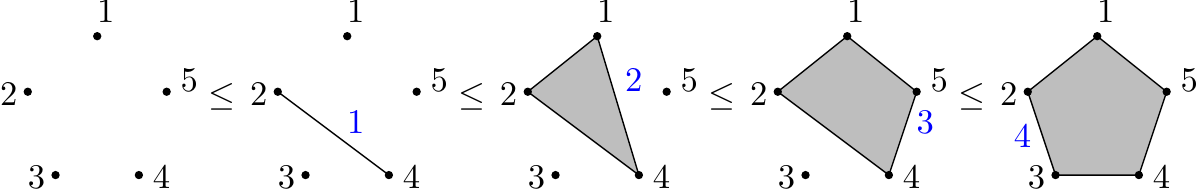}}\\
\includegraphics[scale=.7]{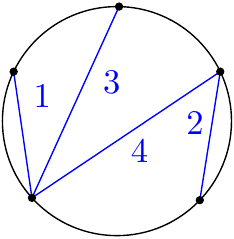} & \raisebox{.3in}{$\longmapsto$} & \raisebox{.1in}{\includegraphics[scale=.7]{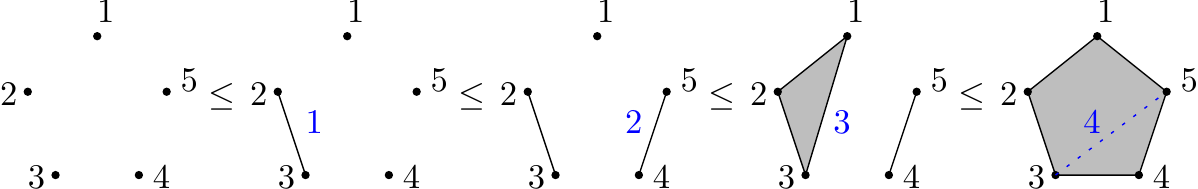}}
\end{array}$$
\end{example}

\section{Future directions}
\label{sec:fut}

In this work, we have only worked out the combinatorics of exceptional sequences associated to the linearly-ordered $\mathbb{A}_n$ Dynkin diagram. However, exceptional sequences are defined for any acyclic quiver. We have made progress in extending our work to the following quiver of type $\DD_n$ and would like to extend our results to all acyclic orientations of simply-laced Dynkin diagrams.
\[\begin{xy} 0;<1pt,0pt>:<0pt,-1pt>:: 
(120,20) *+{} ="0",
(60,20) *+{} ="1",
(90,20) *+{\cdots} ="2",
(30,20) *+{2} ="3",
(0,20) *+{1} ="4",
(160,20) *+{n-2} ="5",
(200,0) *+{n} ="6",
(210,40) *+{n-1} ="7",
"5", {\ar"0"},
"1", {\ar"3"},
"3", {\ar"4"},
"6", {\ar"5"},
"7", {\ar"5"},
\end{xy}\]

We have characterized the posets $\mathcal{P}_d$ that arise from diagrams $d \in \mathcal{D}$ and have shown that the number of linear extensions of $\mathcal{P}_d$ equals the number of complete exceptional sequences that can be formed from the CEC $\overline{\xi}$ corresponding to $d$. In forthcoming work (see \cite{gm2}), we will present a method for enumerating the linear extensions of the posets $\mathcal{P}_d$ and explain the connections between such linear extensions and the theory of factorizations of the long cycle $(1,2,\ldots, n+1) \in \mathfrak{S}_{n+1}$ by a set of transpositions (see \cite{gy02}).

\nocite{*}
\bibliographystyle{alpha}
\bibliography{cmes_a-gm}
\label{sec:biblio}

\end{document}